\documentclass[reqno,10pt]{amsart}
\usepackage{amsmath,amsfonts,amssymb,amsxtra,latexsym,amscd,enumerate,enumitem,amsthm,verbatim}

\usepackage{graphicx}
\usepackage{cancel}

\usepackage{mathtools}
\mathtoolsset{showonlyrefs=true}

\usepackage[dvipsnames]{xcolor}

\usepackage{hyperref}
\hypersetup{colorlinks=true, pdfstartview=FitV,linkcolor=BrickRed,citecolor=BrickRed, urlcolor=BrickRed}
\definecolor{labelkey}{rgb}{0.6,0,0}

\usepackage[margin=1.2in]{geometry}
\numberwithin{equation}{section}

\providecommand{\ip}[1]{\langle#1\rangle}
\providecommand{\abs}[1]{\left\lvert#1\right\rvert}
\providecommand{\norm}[1]{\left\|#1\right\|}

\newcommand{\K}{\mathcal{K}}

\def\R{\mathbb{R}}

\def\K{\mathcal{K}}

\def\L{\mathcal{L}}

\def\N{\mathbb{N}}

\def\eps{\varepsilon}

\newtheorem{theorem}{Theorem}[section]
\newtheorem{lemma}[theorem]{Lemma}
\newtheorem{corollary}[theorem]{Corollary}

\newtheorem{proposition}[theorem]{Proposition}

\theoremstyle{remark}
\newtheorem{remark}[theorem]{Remark}

\title[Scattering map for Vlasov-Poisson]{Scattering map for the Vlasov-Poisson system}

\author{Patrick Flynn}
\address{Brown University}
\email{patrick\_flynn1@brown.edu}

\author{Zhimeng Ouyang}
\address{Brown University}
\email{zhimeng\_ouyang@brown.edu}

\author{Benoit Pausader}
\address{Brown University}
\email{benoit\_pausader@brown.edu}

\author{Klaus Widmayer}
\address{\'Ecole Polytechnique F\'ed\'erale de Lausanne}
\email{klaus.widmayer@epfl.ch}

\begin{document}
\begin{abstract}

We construct (modified) scattering operators for the Vlasov-Poisson system in three dimensions, mapping small asymptotic dynamics as $t\to -\infty$ to asymptotic dynamics as $t\to +\infty$. The main novelty is the construction of modified wave operators, but we also obtain a new simple proof of modified scattering.

 Our analysis is guided by the Hamiltonian structure of the Vlasov-Poisson system. Via a pseudo-conformal inversion we recast the question of asymptotic behavior in terms of local in time dynamics of a new equation with singular coefficients which is approximately integrated using a generating function.
\end{abstract}

\maketitle

\section{Introduction}
The three dimensional Vlasov-Poisson system describes the evolution of a particle distribution\footnote{To be precise, the physically relevant quantity $f(t,x,v)=\mu^2(t,x,v)$ is the square of our unknown $\mu$ -- see also \cite{IPWW2020}.} $\mu(t,x,v):\R\times\R^3\times\R^3\to\R$ satisfying
\begin{equation}\label{VP}
\begin{split}
\left(\partial_t+v\cdot\nabla_x\right)\mu+\lambda\nabla_x\psi\cdot\nabla_v\mu=0,\qquad\Delta_x\psi(t,x)=\rho(t,x),\qquad \rho(t,x)=\int_{\mathbb{R}^3}\mu^2(t,x,v) dv.
\end{split}
\end{equation}
This is a model for a continuum limit of a classical many-body problem with Newtonian self-interactions through a force field  $\nabla_x\psi$ that can be attractive ($\lambda=-1$) as in a galactic setting, or repulsive ($\lambda=1$) as in a plasma or ion gas, and which is generated by the spatial density $\rho(t,x)$ of the particle distribution.

The mathematical theory for the initial value problem associated to \eqref{VP} is classical and guarantees the global existence of unique solutions under suitable assumptions on the initial data \cite{BD1985,LP1991,Pfa1992,Sch1991}. In recent years there has been progress in understanding the long time asymptotic behavior: sharp decay rates of the density and force field are known in some settings \cite{HRV2011,IR1996,Pan2020,Perthame1996,Smu2016,Wan2018}, and it has been shown that for sufficiently small initial data $\mu_0$ the problem \eqref{VP} exhibits a modified scattering dynamic \cite{CK2016,IPWW2020} defined in terms of a limit distribution $\mu_\infty$ and an asymptotic force field $E_\infty[\mu_\infty]$, defined by inverting the roles of $x$ and $v$:
\begin{equation}\label{DefEInftyField}
\begin{split}
 E_\infty[\mu](v)&:= \frac{1}{4\pi}\iint\frac{v-w}{\vert v-w\vert^3}\cdot\mu^2(y,w)\,dydw.
\end{split}
\end{equation}

In this paper, using pseudo-conformal inversion, we prove the converse statement, namely that any solution of the asymptotic dynamic arises in a unique way as a limit of a solution to \eqref{VP}, i.e. we construct the wave operator $\mu_\infty\mapsto\mu_0$. Thus we obtain the existence of a scattering operator linking the asymptotic behavior in the past to the asymptotic behavior in the future ($\mu_{-\infty}\mapsto\mu_0\mapsto\mu_\infty$). 

Our main results can be summarized as follows:
\begin{theorem}\label{thm:main_simple}
 There exists $\varepsilon>0$ such that:
\begin{enumerate}[label=(\roman*)]
\item\label{it:modscat} (Global existence and modified scattering) Given $\mu_1(x,v)$ satisfying
  \begin{equation}\label{eq:modscat_assump}
   \norm{\mu_1}_{L^2_{x,v}}+\norm{\ip{x-v}^2\mu_1}_{L^\infty_{x,v}}+ \norm{\nabla_{x,v}\mu_1}_{L^\infty_{x,v}}\leq \eps,
  \end{equation}
  there exists a unique global strong solution $\mu$ of the initial value problem for \eqref{VP} with $\mu(1,x,v)=\mu_1(x,v)$. In addition, there exist $\mu_{\infty}(x,v)$ and $E_{\infty}=E_\infty[\mu_{\infty}]$ as in \eqref{DefEInftyField} such that, locally uniformly in $(x,v)$,
  \begin{equation}\label{eq:modscat}
   \mu(t,x+tv-\lambda \ln( t)E_{\infty}(v),v)\to \mu_{\infty}(x,v),\qquad t\to+ \infty.
  \end{equation}

\item\label{it:waveops} (Existence of modified wave operators) Given $\mu_\infty\in W^{2,\infty}(\R^3_a\times\R^3_b)$ and $E_\infty=E_\infty[\mu_\infty]\in W^{3,\infty}(\R^3)$ as in \eqref{DefEInftyField} satisfying
  \begin{equation}\label{eq:waveop_assump}
   \norm{\mu_\infty}_{L^2_{a,b}}+\norm{\ip{a}^{5}\mu_\infty}_{L^\infty_{a,b}}+\Vert \langle a\rangle\nabla_{a,b}\mu_\infty\Vert_{L^\infty_{a,b}}+\Vert\langle a\rangle^2\nabla^2_{a,b}\mu_\infty\Vert_{L^\infty_{a,b}}+\norm{E_\infty}_{W^{3,\infty}}<\infty,
  \end{equation}
  there exists a unique strong global solution $\mu$ of \eqref{VP} for which \eqref{eq:modscat} holds.
  
\item\label{it:ScatOp} (Scattering map) For any asymptotic state $\mu_{-\infty}$ with $E_\infty[\mu_{-\infty}]\in W^{3,\infty}(\mathbb R^3)$ as in \eqref{DefEInftyField},
\begin{equation*}
\Vert \mu_{-\infty}\Vert_{L^2_{a,b}}+\Vert\langle a,b\rangle^5\mu_{-\infty}\Vert_{L^\infty_{a,b}}+\Vert\langle a\rangle\nabla_{a,b}\mu_{-\infty}\Vert_{L^\infty_{a,b}}+\Vert\langle a\rangle^2\nabla_{a,b}^2\mu_{-\infty}\Vert_{L^\infty_{a,b}}\le \varepsilon,
\end{equation*}
 there exists a unique strong solution $\mu$ of \eqref{VP}, and $\mu_{\infty}\in L^2_{a,b}\cap L^\infty_{a,b}$ such that 
 \begin{equation}\label{eq:waveops}
  \mu(t,x+tv\mp\lambda\ln(\langle t\rangle )E_{\pm\infty}(v),v)\to \mu_{\pm\infty}(x,v),\qquad t\to\pm\infty.
 \end{equation}
\end{enumerate}
\end{theorem}

We call the map defined in a neighborhood of the origin in the Schwarz space through \ref{it:ScatOp} above,
\begin{equation}\label{ScatteringMap}
\mathcal{S}:\mu_{-\infty}\mapsto\mu_{\infty}
\end{equation}
 the {\it Scattering map}. We refer to Theorem \ref{ModScatThm} for a more precise statement of our results for \ref{it:modscat} and to Theorem \ref{WaveOpsThm} for a more precise statement of \ref{it:waveops}. In particular we note that the force field has optimal decay $\vert \nabla\psi\vert\lesssim \langle t\rangle^{-2}$ in all cases.

\begin{remark}\label{rem:intro}
We comment on some points:
\begin{enumerate}[leftmargin=*]

\item The main novelty of this work is the construction of the wave operator \ref{it:waveops}. While the small data modified scattering dynamic \eqref{eq:modscat} was already obtained in \cite{IPWW2020}, the present result \ref{it:modscat} is also of interest since it is stronger and the approach, while less generalizable, leads to a simple derivation of the asymptotic dynamic. We also refer to \cite{PW2020} for yet another point of view on the modified scattering as arising from mixing.

\item Our topology for small data/modified scattering in \eqref{eq:waveop_assump} is weaker than in all other works on asymptotic behavior that we are aware of \cite{BD1985,CK2016,HRV2011,IPWW2020,Smu2016,Wan2018,Pan2020}. It is unclear what the optimal topology is, but to get almost Lipschitz bounds on the force field, by \eqref{CE}, one cannot work in a much weaker setting than ours.

\item We also obtain propagation of regularity: assuming more regularity on the initial data we obtain higher regularity on the final (scattering) data and vice-versa.

\item\label{it:inf_conslaws} Our initial data for scattering may have infinite energy and momentum; in addition, a simple modification also allows for initial data of infinite mass. It is unclear which role (if any) the physical conservation laws play for the asymptotic behavior.

\item It is worth noting a curious fact: our proof can be adapted directly to the case of a plasma of two species (ions and electrons). In this case, using \ref{it:waveops}, one can construct solutions for which the asymptotic electric field profile $E_\infty\equiv 0$ vanishes and the solutions {\it scatter linearly}. In this case, the same equation allows two different asymptotic behaviors. It remains to be understood to which extent the linear scattering is nongeneric (say in case the total charge vanishes).
\end{enumerate}
\end{remark}

\subsubsection*{About the Proof of Theorem \ref{thm:main_simple}}

In the spirit of the prior work \cite{IPWW2020} (see also \cite{LMR2008,LMR2008-JAMS}), we build on parallels between kinetic and dispersive equations. In particular, the Hamiltonian structure of \eqref{VP} guides our analysis.

The simplest case for asymptotic behavior of a nonlinear equation is {\it linear scattering} when the nonlinearity can simply be neglected to model asymptotic dynamics. For the Vlasov-Poisson system, this happens in the setting of Landau damping \cite{BMM2016,GNR2020,MV2011}, the ion/screened problem \cite{BMM2018,HNR2019}, and in higher dimensions \cite{Smu2016}, where solutions asymptotically satisfy $\mathcal{T}(\mu)=0$ with $\mathcal{T}$ defined in \eqref{FreeStreaming}.
The asymptotic behavior of \emph{modified scattering} as in \eqref{eq:modscat} and \eqref{eq:waveops} can be viewed as a manifestation of the unrelenting relevance of nonlinear interactions in \eqref{VP} throughout time. In \eqref{VP} the nonlinear, long-range interactions are governed by a force field which does not decay fast enough to produce only a finite correction as time tends to infinity and produces the logarithmic corrections identified in the above theorem -- see also \cite{CK2016,IPWW2020,PW2020} for the Vlasov-Poisson setting, and \cite{HPTV2015,HN1998,IT2015,KP2011,Ouy2020} for related results on other equations.

In order to understand the asymptotic behavior, we need to $(i)$ identify a mechanism for decay (here dispersion), $(ii)$ prove global existence, $(iii)$ isolate an asymptotic dynamic and $(iv)$ prove convergence to it. We offload the dispersion to the pseudo-conformal transform $\mathcal{I}$ which compactifies time and reduces global existence to local existence for a singular equation  -- see also \cite{Bou1999,CN2018,Chr1986,Tao2009} for similar ideas. At this point, the problem merely reduces to establishing convergence at the image of infinity, $s=0$, where,  however, the equation has a violent singularity. We extend the force field $E=-\nabla\psi$ via a variant of the continuity equation:
\begin{equation}\label{CE}
\partial_tE+\nabla\Delta^{-1}\hbox{div}({\bf j})=0,\qquad {\bf j}(t,q)=\int p\gamma^2(t,q,p)dp,
\end{equation}
which does not involve the (singular) acceleration and provides good control of $E$ so long as we control some moments of $\gamma$. Once we obtain convergence of $E$ to a fixed asymptotic field $E_0$, the equation becomes a simple  perturbation of transport by a shear term:
\begin{equation*}
\begin{split}
\left(\partial_s+\lambda s^{-1}E_0(q)\cdot\nabla_p\right)\gamma=O(1),
\end{split}
\end{equation*}
which is easily integrated to recover the dynamic originally isolated in \cite{IPWW2020}. To make this rigorous, we need to propagate mild control on appropriate norms. This is done through a bootstrap that allows some deterioration over time in different ways depending on the scenarios: growth of nonconvergent norms in the case of modified scattering and loss of moment in the case of wave operators (where we start from the singular time $s=0$).

The proof of part \ref{it:modscat} shows how natural the pseudo-conformal inversion $\mathcal{I}$ is to study asymptotics of \eqref{VP}: working with only moments that are conserved in the linear evolution of \eqref{VP} one directly obtains global solutions in a bootstrap argument. Additional regularity as in \eqref{eq:modscat_assump} is easily propagated to yield unique strong solutions and to recover the asymptotic behavior \eqref{eq:modscat} -- see Section \ref{sec:mod_scat}.

Part \ref{it:waveops} is proved using a canonical change of variables in \eqref{NewVP} to mitigate the strong singularity at $s=0$ -- see Section \ref{sec:wave_ops}. The Cauchy problem for the resulting equations \eqref{eq:sigma} can in fact be (locally) solved starting from $s=0$ for a sufficiently large class of initial data as in \eqref{eq:waveop_assump}. Again, moments are easily bootstrapped, while propagating derivatives requires to identify a proper weighted norm which compensates for the {\it ill-conditioned} Hessian of the new Hamiltonian by allowing one loss of moment. Since via $\mathcal{I}$ this corresponds to a strong solution on $[T,\infty)$ for some $T>0$, classical theory as in \cite{LP1991} then gives a global solution.

Finally \ref{it:ScatOp} follows simply by combining \ref{it:waveops} (backwards in time) to go from past-asymptotic data to initial data and \ref{it:modscat} to go from initial data to future asymptotic data.

\subsubsection*{Open Questions} We list some open questions which remain outstanding:

\begin{itemize}

\item Is there a topology that makes the scattering operator in \eqref{ScatteringMap} an endomorphism?

\item In the plasma case $\lambda=+1$, what is the asymptotic behavior for large data? Solutions are global, there are no nontrivial equilibriums and the wave operators are defined for large data, so it is tempting to believe that Theorem \ref{thm:main_simple} may be extended to all solutions (see \cite{Pan2020} for radial data).

\item In the gravitational case $\lambda=-1$, is there a ``ground state'', i.e.\ a smallest solution which does not scatter? Are there solutions which satisfy some form of modified scattering towards a nonzero stationary solution? This appears very challenging, but we note \cite{PW2020} for an example of stability result around a nonzero equilibrium in a related setting and \cite{DSS2004} for related works.

\end{itemize}

\subsection{Pseudo-Conformal Inversion}\label{ssec:pseudo-conf}

We define the involution of $\mathbb{R}\times\mathbb{R}^3\times\mathbb{R}^3$ given by the \emph{pseudo-conformal inversion} (see also \cite{LMR2008})
\begin{equation}\label{PCT}
\begin{split}
 \mathcal{I}:(t,x,v)\mapsto (1/t,x/t,x-tv).
\end{split}
\end{equation}
This transformation interacts favorably with free streaming,
\begin{equation}\label{FreeStreaming}
\mathcal{T}:=\partial_t+v\cdot\nabla_x,
\end{equation}
since heuristically it exchanges the role of $v$ with that of $x-tv$, both of which are conserved along the evolution (i.e.\ commute with $\mathcal{T}$). Indeed, one can observe that if $(s,q,p)=\mathcal{I}(t,x,v)$,
\begin{equation*}\label{NewDerivatives}
\begin{split}
\partial_s=-s^{-2}\left(\partial_t+q\cdot\nabla_x\right)-p\cdot\nabla_v,\qquad \nabla_q=s^{-1}\nabla_x+\nabla_v,\qquad\nabla_p=-s\nabla_v,
\end{split}
\end{equation*}
and
\begin{equation*}
 \mathcal{T}(f\circ \mathcal{I})=-s^{-2}\mathcal{T}(f)\circ \mathcal{I}.
\end{equation*}
so that composition with $\mathcal{I}$ preserves the class of solutions of free streaming $\mathcal{T}f=0$. The transformation $\mathcal{I}$ is almost symplectic in the sense that $dq\wedge dp=-dx\wedge dv$, and in particular the total charge is preserved:
\begin{equation*}
\begin{split}
\iint  (f\circ\mathcal{I})^2dqdp=\iint f^2dxdv.
\end{split}
\end{equation*} 

\subsubsection*{Recasting Vlasov-Poisson}

Given a solution $\mu(t,x,v)$ of \eqref{VP}, we let $\gamma=\mu\circ \mathcal{I}$, so that
\begin{equation}\label{TransformationMuGamma}
\begin{split}
\gamma(s,q,p)&:=\mu(\frac{1}{s},\frac{q}{s},q-sp),\qquad\mu(t,x,v)=\gamma(\frac{1}{t},\frac{x}{t},x-tv).
\end{split}
\end{equation}
 The Vlasov-Poisson system involves a perturbation of free streaming \eqref{FreeStreaming} by a force field (in this paper, we stick to the plasma terminology and refer to it as the ``Electric field''):
 \begin{equation}\label{DefEField}
 \begin{split}
 E[\mu](t,x)&:=  \nabla_x \Delta^{-1}_x \int \mu(t,x,v)^2 dv = \frac{1}{4\pi}\iint\frac{x-y}{\vert x-y\vert^3}\cdot\mu^2(t,y,v)dvdy,
 \end{split}
 \end{equation}
which also transforms naturally:
\begin{equation*}
\begin{split}
E[\mu](t,tx)&=\frac{1}{t^{2}}E[\gamma](\frac{1}{t},x),
\end{split}
\end{equation*}
and we see that $\mu$ solves \eqref{VP} on $0\le T_\ast\le t\le T^\ast$ if and only if $\gamma$ satisfies for $0\le (T^\ast)^{-1}\le s\le (T_\ast)^{-1}$
\begin{equation}\label{NewVP}
\begin{split}
\left(\partial_s+p\cdot\nabla_q\right)\gamma+\lambda s^{-1}E[\gamma]\cdot\nabla_p\gamma=0.
\end{split}
\end{equation}

\bigskip
\section{The Force Field and the Continuity Equation}\label{SecE}
To prove both the modified scattering and wave operator theorems, we require general estimates on the electric field $E$ defined in \eqref{DefEField}. In Lemma \ref{LemControlEF}, we prove fix-time bounds on the operator $\gamma\mapsto E$ and in Lemma \ref{lem:Ediff} we obtain dynamic bounds for an electric field which satisfies a slight strengthening of the continuity equation \eqref{CE}, namely:
\begin{equation}\label{ContinuityEquation}
\partial_s\left\{\gamma^2\right\}+\hbox{div}_q\left\{p\gamma^2\right\}+\hbox{div}_p\{F\gamma^2\}=0
\end{equation}
for some force-field $F$.
 
\begin{lemma}\label{LemControlEF}
Let $\gamma = \gamma(q,p)$ be such that
$\gamma \in L^2_{q,p}$, $\langle p\rangle^{2} \gamma \in L^\infty_{q,p}$ and $\nabla_q\gamma \in L^\infty_{q,p}$ and $E=E[\gamma]$ defined by \eqref{DefEField}.
For all $A>0$ and $\kappa\in(0,1/3)$ we have
\begin{equation}\label{ControlEF}
\begin{split}
\Vert E\Vert_{L^\infty_q}&\lesssim A\left[\Vert\gamma\Vert_{L^2_{q,p}}^2+\Vert\gamma\Vert_{L^\infty_{q,p}}^2\right]+A^{-1}\Vert \vert p\vert^2\gamma\Vert_{L^\infty_{q,p}}^2,\\
\Vert \nabla_qE(s)\Vert_{L^\infty_q}&\lesssim A \Vert \gamma\Vert_{L^2_{q,p}}^2+ A^{-\frac{\kappa}{3}}\Vert |p|^2\gamma\Vert_{L^\infty_{q,p}}^2+ A^{\kappa - \frac{1}{3}}\Vert \gamma\Vert_{L^\infty_{q,p}}\Vert\nabla_q\gamma\Vert_{L^\infty_{q,p}}.
\end{split}
\end{equation}

\end{lemma}

In fact, we will mostly make use of the second line of \eqref{ControlEF} corresponding to the choice $A=\ip{\ln(s)}^4$, $\kappa=\frac{1}{30}$, i.e.\ the bound
\begin{equation}\label{ControlEF'}
 \Vert \nabla_qE(s)\Vert_{L^\infty_q}\lesssim \ip{\ln(s)}^4 \Vert \gamma\Vert_{L^2_{q,p}}^2+ \Vert |p|^2\gamma\Vert_{L^\infty_{q,p}}^2+ \ip{\ln(s)}^{-\frac{6}{5}}\Vert \gamma\Vert_{L^\infty_{q,p}}\Vert\nabla_q\gamma\Vert_{L^\infty_{q,p}}.
\end{equation}

\begin{remark}
In the estimates of this section, up to minor modifications one may alternatively work with the $\langle p\rangle^{-1}L^4_{q,p}$ norm of $\gamma$, rather than its $L^2_{q,p}$ norm. This allows to consider initial data with infinite mass -- see also Remark \ref{rem:intro} \eqref{it:inf_conslaws}.
\end{remark}

\begin{proof}[Proof of Lemma \ref{LemControlEF}]
We decompose the electric field on different scales using a radially symmetric function $\chi\in C^\infty_c(\{1/2\le \vert y\vert\le 2\})$ with $\int_{\mathbb{R}^3} \chi(y)dy =1$, namely
\begin{equation*}
\begin{split}
E^j[\gamma](q)=c\int_{R=0}^\infty  E^j_R(q)\frac{dR}{R^2},\qquad E^j_R[\gamma](q):=\iint R^{-1}\{\partial_{q^j}\chi\}(R^{-1}(q-r))\cdot\gamma^2(r,u)\, dr du,
\end{split}
\end{equation*}
and we directly obtain the following elementary bounds
\begin{equation}\label{ElementaryE}
\begin{split}
 E^j_R&\lesssim R^{-1}\Vert\gamma\Vert_{L^2_{q,p}}^2,\qquad \vert\partial_qE^j_R\vert\lesssim R^{-2}\Vert\gamma\Vert_{L^2_{q,p}}^2,
\end{split}
\end{equation}
which is enough for large $R$. To go further, we introduce
\begin{equation*}
\begin{split}
E^j_{R,V}[\gamma](q):=\iint  R^{-1}\{\partial_{q^j}\chi\}(R^{-1}(q-r))\cdot\chi(V^{-1}u)\cdot\gamma^2(r,u) \, dr du,
\end{split}
\end{equation*}
with $E^j[\gamma](q)=c\int_{R=0}^\infty\int_{V=0}^\infty  E^j_{R,V}(q)\frac{dV}{V}\frac{dR}{R^2}$ and we estimate
\begin{equation}\label{ElementaryE2}
\begin{split}
\vert E^j_{R,V}\vert&\lesssim R^2\min\{V^3\Vert \gamma\Vert_{L^\infty_{q,p}}^2,V^{-1}\Vert |p|^2\gamma\Vert_{L^\infty_{q,p}}^2\},\\
\vert\partial_qE^j_{R,V}\vert&\lesssim R\min\{V^{-1}\Vert |p|^2\gamma\Vert_{L^\infty_{q,p}}^{2},RV^3\Vert \nabla_q\gamma\Vert_{L^\infty_{q,p}}\Vert\gamma\Vert_{L^\infty_{q,p}}\}.
\end{split}
\end{equation}
From this we deduce that
\begin{equation*}
\begin{split}
\vert E^j[\gamma]\vert&\lesssim \int_{R=A}^\infty \vert E^j_R\vert\frac{dR}{R^2}+\int_{R=0}^A\int_{V=0}^{B}\vert E^j_{R,V}\vert\frac{dR}{R^2}\frac{dV}{V}+\int_{R=0}^A\int_{V=B}^{\infty}\vert E^j_{R,V}\vert\frac{dR}{R^2}\frac{dV}{V}\\
&\lesssim A^{-2}\Vert\gamma\Vert_{L^2_{q,p}}^2+AB^3\Vert \gamma\Vert_{L^\infty_{q,p}}^2+AB^{-1}\Vert \vert p\vert^2\gamma\Vert_{L^\infty_{q,p}}^2
\end{split}
\end{equation*}
and choosing $A=B^{-1}$, we obtain the first line of \eqref{ControlEF}. Similarly, we see that for $\kappa\in(0,1/3)$
\begin{equation*}
\begin{split}
\vert \partial_qE^j\vert&\lesssim \int_{R=A}^\infty \vert\partial_qE^j_R\vert\frac{dR}{R^2}+\int_{R=0}^A\int_{V=0}^{R^{-\kappa}}\vert\partial_{q}E^j_{R,V}\vert\frac{dV}{V} \frac{dR}{R^2}+\int_{R=0}^A\int_{V=R^{-\kappa}}^{\infty}\vert\partial_{q}E^j_{R,V}\vert \frac{dV}{V}\frac{dR}{R^2}\\
&\lesssim A^{-3}\Vert\gamma\Vert_{L^2_{q,p}}^2+A^\kappa \Vert |p|^2\gamma\Vert_{L^\infty_{q,p}}^2+A^{1-3\kappa}\Vert\gamma\Vert_{L^\infty_{q,p}}\Vert\nabla_q\gamma\Vert_{L^\infty_{q,p}}.
\end{split}
\end{equation*}
After substituting $A$ with $A^{-1/3}$, this gives the second line of \eqref{ControlEF}.
\end{proof}

\begin{lemma}\label{lem:Ediff} 
Fix $0 < s_0 < s_1$. Suppose $\gamma \in L^\infty_s([s_0,s_1]; L^2_{q,p})$ satisfies \eqref{ContinuityEquation} in the sense of distribution for some bounded force field $F$.

\noindent (i) we can find two variants of the Lipschitz norm (in time) of $E$:
\begin{equation}\label{eq:Ediff}
\begin{aligned}
\|E(s_1)-E(s_0)\|_{L^\infty_q} &\lesssim \langle \ln(s_1-s_0) \rangle^2 (s_1 - s_0)\Vert \vert p\vert^2\gamma\Vert_{L^\infty_{s,q,p}}^2+(t_1-t_0)^2\left[\Vert \gamma\Vert_{L^\infty_{s,q,p}}^2+\Vert \gamma\Vert_{L^\infty_sL^2_{q,p}}^2\right]\\
&\quad+(s_1-s_0)^3\langle\ln(s_1/s_0)\rangle\Vert \langle p\rangle^2\gamma\Vert_{L^\infty_{s,q,p}}^2\Vert sF\Vert_{L^\infty_{s,q,p}}
\end{aligned}
\end{equation}
and assuming only \eqref{CE}, we see that
\begin{equation}\label{eq:EdiffVar}
\begin{aligned}
\|E(s_1)-E(s_0)\|_{L^\infty_q} &\lesssim \langle \ln(s_1-s_0) \rangle (s_1 - s_0)\Vert {\bf j}\Vert_{L^\infty_{s,q}}+(s_1-s_0)^2\left[\Vert \langle p\rangle^2\gamma\Vert_{L^\infty_{s,q,p}}^2+\Vert \gamma\Vert_{L^\infty_sL^2_{q,p}}^2\right],
\end{aligned}
\end{equation}
\noindent (ii) we also have the corresponding estimate for $\nabla_q E$:
\begin{equation}\label{eq:EdiffVarGrad}
\begin{aligned}
\|\nabla E(s_1)-\nabla E(s_0)\|_{L^\infty_q} &\lesssim \langle \ln(s_1-s_0) \rangle (s_1 - s_0)\Vert \nabla_q{\bf j}\Vert_{L^\infty_{s,q}}\\
&\quad +(s_1-s_0)^2\left[\Vert \langle p\rangle^4\gamma\Vert_{L^\infty_{s,q,p}}^2+\Vert \nabla_q\gamma\Vert_{L^\infty_{s,q,p}}^2+\Vert \gamma\Vert_{L^\infty_sL^2_{q,p}}^2\right],
\end{aligned}
\end{equation}
from which we deduce
\begin{equation}\label{eq:gradEdiff}
\begin{aligned}
\|\nabla E(s_1)-\nabla E(s_0)\|_{L^\infty_q} &\lesssim \langle \ln(s_1-s_0) \rangle (s_1 - s_0)\left[\Vert \langle p\rangle^5\gamma\Vert_{L^\infty_{s,q,p}}^2+\Vert \nabla_q\gamma\Vert_{L^\infty_{s,q,p}}^2+\Vert \gamma\Vert_{L^\infty_sL^2_{q,p}}^2\right]
\end{aligned}
\end{equation}

\end{lemma}

\begin{proof}

$(i)$ Using \eqref{ElementaryE} and \eqref{ElementaryE2} we see that for $s\in\{s_0,s_1\}$,
\begin{equation}\label{SmallERVar1}
\begin{split}
\int_{R=A^{-1}}^\infty \vert E_R(s)\vert \frac{dR}{R^2}&\lesssim A^{2}\Vert \gamma(s)\Vert_{L^2_{q,p}}^2,\qquad \int_{R=0}^{A^2} \vert E_{R}(s)\vert \frac{dR}{R^2}\lesssim A^2\Vert \langle p\rangle^2\gamma(s)\Vert_{L^\infty_{q,p}}^2,
\end{split}
\end{equation}
and
\begin{equation*}
\begin{split}
\int_{R=0}^{A^{-1}}\int_{V=0}^{B} \vert E_{R,V}(s)\vert \frac{dR}{R^2}\frac{dV}{V}&\lesssim A^{-1}B^3\Vert\gamma(s)\Vert_{L^\infty_{q,p}}^2,\\
\int_{R=0}^{A^{-1}}\int_{V=B^{-3}}^\infty \vert E_{R,V}(s)\vert \frac{dR}{R^2}\frac{dV}{V}&\lesssim A^{-1}B^3\Vert \vert p\vert^2\gamma(s)\Vert_{L^\infty_{q,p}}^2,
\end{split}
\end{equation*}
and we conclude that
\begin{equation*}
\begin{split}
\vert E(s)-\int_{R=A^2}^{A^{-1}}E_{R,a}(s)\frac{dR}{R^2}\vert&\le A^{-1}B^3\Vert \langle p\rangle^2\gamma\Vert_{L^\infty_{q,p}}^2+A^2\left[\Vert \gamma\Vert_{L^2_{q,p}}^2+\Vert \langle p\rangle^2 \gamma\Vert_{L^\infty_{q,p}}^2\right],\\
E_{R,a}&:=\iint R^{-1}\{\partial_{q^j}\chi\}(R^{-1}(q-r))\cdot\chi_{\{B\le\cdot\le B^{-3}\}}(u)\cdot\gamma^2(r,u)\, dr du,
\end{split}
\end{equation*}
where
\begin{equation*}
\chi_{\{B\le\cdot\le B^{-3}\}}(u)=\int_{\{B\le V\le B^{-3}\}}\chi(V^{-1}u)\frac{dV}{V}.
\end{equation*}
On the other hand, using the equation \eqref{ContinuityEquation} we find that
\begin{equation}\label{UsingContinuityEquation}
\begin{split}
 0&=\int_{s=s_0}^{s_1}\iint  R^{-1}\{\partial_{q^j}\chi\}(R^{-1}(q-r))\cdot\chi_{\{B\le\cdot\le B^{-3}\}}(u)\cdot \left\{\partial_s\gamma^2+\hbox{div}_{r}(\gamma^2u)+\hbox{div}_u(F\gamma^2)\right\} dr du ds,\\
 &=E_{R,a}(s_1)-E_{R,a}(s_0)+\int_{s=s_0}^{s_1}\iint R^{-2}u^k\{\partial_{q^j}\partial_{q^k}\chi\}(R^{-1}(q-r))\cdot\chi_{\{B\le\cdot\le B^{-3}\}}(u)\cdot\gamma^2(s,r,u)\, dr du ds,\\
&\qquad-\int_{s=s_0}^{s_1}\iint R^{-1}\partial_{q^j}\chi(R^{-1}(q-r))\cdot\gamma^2(s,r,u) \cdot (F\cdot\nabla_u)\chi_{\{B\le\cdot\le B^{-3}\}}(u)\, dr du ds.
\end{split}
\end{equation}
Since
\begin{equation*}
\begin{split}
\left\vert\nabla_u \chi_{\{B\le\cdot\le B^{-3}\}}(u)\right\vert&\lesssim B^{-1}\mathfrak{1}_{\{\vert u\vert\le 2B\}}+B^{3}\mathfrak{1}_{\{\vert u\vert\ge B^{-3}/2\}},
\end{split}
\end{equation*}
we see that
\begin{equation*}
\begin{split}
&\left\vert \iint R^{-1}\partial_{q^j}\chi(R^{-1}(q-r))\cdot\gamma^2(r,u) \cdot (F\cdot\nabla_u)\chi_{\{B\le\cdot\le B^{-3}\}}(u) dr du\right\vert\\
\lesssim &\Vert s F\Vert_{L^\infty_{q,p}}\cdot s^{-1}R^2\cdot\left[B^{2}\Vert \gamma\Vert_{L^\infty_{q,p}}^2+B^6\Vert \vert p\vert^2\gamma\Vert_{L^\infty_{q,p}}^2\right]
\end{split}
\end{equation*}
and using a crude bound for the second integral in \eqref{UsingContinuityEquation}, we find that
\begin{equation*}
\begin{split}
\vert \int_{R=A^2}^{A^{-1}}\left\{E_{R,a}(s_1)-E_{R,a}(s_0)\right\}\frac{dR}{R^2}\vert&\lesssim (s_1-s_0)\cdot \Vert \vert u\vert^2\gamma\Vert_{L^\infty_{s,r,u}}^2\cdot\int_{R=A^2}^{A^{-1}}\frac{dR}{R}\cdot \int_{V=B}^{B^{-3}}\frac{dV}{V}\\
&\quad+\langle\ln(s_1/s_0)\rangle\cdot \Vert sF\Vert_{L^\infty_{s,q,p}}\cdot A^{-1}B^{2}\Vert \langle p\rangle^2\gamma\Vert_{L^\infty_{s,q,p}}^2.
\end{split}
\end{equation*}
Letting $B=A^2=(s_1-s_0)^2$, we obtain the result. For the variant \eqref{eq:EdiffVar}, we do not localize in $u$. In this case, we need only use \eqref{SmallERVar1} and the last term in \eqref{UsingContinuityEquation} simplifies. We detail this in the similar $(ii)$ below.

$(ii)$ We use similar analysis without localizing in $u$. Passing the derivative onto $\gamma$ gives
\begin{equation*}
\begin{split}
\int_{R=A^{-\frac{2}{3}}}^\infty \vert \nabla E_R(s)\vert \frac{dR}{R^2}&\lesssim A^{2}\Vert \gamma(s)\Vert_{L^2_{q,p}}^2,\qquad \int_{R=0}^{A^2} \vert \nabla E_{R}(s)\vert \frac{dR}{R^2}\lesssim A^2\Vert\nabla_q\gamma(s)\Vert_{L^\infty_{q,p}}\cdot \Vert \langle p\rangle^4\gamma(s)\Vert_{L^\infty_{q,p}},\\
\end{split}
\end{equation*}
and the continuity equation \eqref{ContinuityEquation} gives
\begin{equation*}
\begin{split}
 0&=\int_{s=s_0}^{s_1}\iint  R^{-1}\{\partial_{q^j}\chi\}(R^{-1}(q-r))\cdot \partial_j\left\{\partial_s\gamma^2+\hbox{div}_{r}(\gamma^2u)\right\}\, dr du ds\\
 &=\partial_jE_{R}(s_1)-\partial_j E_{R}(s_0)+2\int_{s=s_0}^{s_1}\iint R^{-2}u^k\{\partial_{q^j}\partial_{q^k}\chi\}(R^{-1}(q-r))\cdot\gamma\cdot\nabla_q\gamma (r,u)\, dr du ds,
\end{split}
\end{equation*}
from which we deduce that
\begin{equation*}
\begin{split}
\Vert \nabla E_R(s_1)-\nabla E_R(s_0)\Vert_{L^\infty_{q,p}}&\lesssim (s_1-s_0)\cdot R\cdot \Vert \langle p\rangle^5\gamma\Vert_{L^\infty_{s,q,p}}\Vert\nabla_q\gamma\Vert_{L^\infty_{s,q,p}}
\end{split}
\end{equation*}
and integrating in $A^2\le R\le A^{-1}$, we obtain \eqref{eq:gradEdiff}.
\end{proof}

Finally we collect the modifications of Lemma \ref{LemControlEF} and \ref{lem:Ediff} above needed to consider smoother solutions. The proofs are similar (passing the derivative through the density) and are omitted.
\begin{lemma}\label{ControlEPropReg}
There holds that for all $\kappa \in (0,\frac{1}{3})$,
\begin{equation*}
\begin{split}
\Vert \nabla^2_qE\Vert_{L^\infty_q}&\lesssim  A\|\gamma\|_{L^2_{q,p}}^2+A^{-\frac{\kappa}{4}}\||p|^4 \gamma\|_{L^\infty_{q,p}} \|\nabla_q \gamma\|_{L^\infty_{q,p}}+  A^{\frac{3\kappa - 1}{4}}\|\gamma\|_{L^\infty_{q,p}}\|\nabla^2_q \gamma\|_{L^\infty_{q,p}}
\end{split}
\end{equation*}
and
\begin{equation*}
\begin{split}
\Vert \nabla^2_qE(s_1)-\nabla_q^2E(s_0)\Vert_{L^\infty_q}&\lesssim \langle \ln(s_1 - s_0)\rangle (s_1 - s_0)\Vert \nabla_{q}^2{\bf j}\Vert_{L^\infty_{s,q}}\\
&\quad+ (s_1 - s_0)^2\left[\|\gamma\|_{L^\infty_s L^2_{q,p}}^2+\Vert 
\langle p\rangle^5\gamma\Vert_{L^\infty_{s,q,p}}\Vert\nabla^2_{q}\gamma\Vert_{L^\infty_{s,q,p}}+\Vert \langle p\rangle^{2.1}\nabla_q\gamma\Vert_{L^\infty_{s,q,p}}^2\right].\\
\end{split}
\end{equation*}
\end{lemma}

\bigskip
\section{Modified Scattering}\label{sec:mod_scat}
While we only need to study \eqref{NewVP} on a compact time interval, this equation is now time dependent with a violent singularity at $s=0$. This can be mitigated since the singular terms
\begin{equation*}
\left(\partial_s+\lambda s^{-1}E(s,q)\cdot\nabla_p\right)\gamma=l.o.t.
\end{equation*}
can be integrated to main order:
\begin{equation}\label{DefGAMMA}
\begin{split}
 \Gamma(s,q,p)=\gamma(s,q,p+\lambda\int_{s^\prime=1}^sE(s^\prime,q)\frac{ds^\prime}{s^\prime}),\qquad \gamma(s,q,p)=\Gamma(s,q,p-\lambda\int_{s^\prime=1}^sE(s^\prime,q)\frac{ds^\prime}{s^\prime}).
\end{split}
\end{equation}
Since $\Gamma$ satisfies an equivalent but more cumbersome equation, we prefer to work with \eqref{NewVP} to bootstrap control of the norms, but a variant of \eqref{DefGAMMA} leads quickly to the modified dynamics \eqref{eq:nu-conv} once $E$ is shown to converge.

The main result of this section is the following statement about modified scattering:
\begin{theorem}\label{ModScatThm}
There exists $\varepsilon>0$ such that if $\gamma_1(q,p)$ satisfies
\begin{equation}\label{MSAssID}
 \Vert \gamma_1\Vert_{L^2_{q,p}}+\Vert \langle p\rangle^{2}\gamma_1\Vert_{L^\infty_{q,p}}+\Vert \nabla_{p,q}\gamma_1\Vert_{L^\infty_{q,p}}\le\varepsilon_0\le\varepsilon,
\end{equation}
then there exists a unique solution $\gamma$ of \eqref{NewVP} with ``initial'' data $\gamma(s=1)=\gamma_1$ for all times $0<s\leq 1$, and $\gamma\in L^\infty_s((0,1],L^\infty_{q,p}\cap L^2_{q,p})$ satisfies
\begin{equation}
 \norm{\ip{p}^2\gamma(s)}_{L^\infty_{q,p}}\lesssim \varepsilon_0\ip{\ln(s)}^2,\qquad\norm{\nabla_{p,q}\gamma(s)}_{L^\infty_{q,p}}\lesssim \varepsilon_0\ip{\ln(s)}^5.
\end{equation}
If in addition
\begin{equation}\label{eq:modscat_assump2}
 \norm{\ip{p}\nabla_{p,q}\gamma_1}_{L^\infty_{q,p}}\leq \eps_0,
\end{equation}
then $\norm{\ip{p}\nabla_{p,q}\gamma(s)}_{L^\infty_{q,p}}\lesssim \varepsilon_0\ip{\ln(s)}^6$ and there exist $E_0=E[\gamma_0]\in L^\infty_{q,p}$ and $\gamma_0\in L^\infty_{q,p}$ such that, uniformly in $q,p$,
\begin{equation}\label{eq:nu-conv}
 \gamma(s,q+ps+\lambda s\ln(s)E_0(q),p+\lambda\ln(s)E_0(q))\to\gamma_0(q,p),\qquad s\to 0.
\end{equation}
\end{theorem}

\begin{remark}\label{RemarkModScatThm}
We comment on some points of interest:
\begin{enumerate}
 \item In fact, as we will show below one can obtain global solutions in a bootstrap argument involving only the moments $\ip{p}^2\gamma$. The higher regularity of \eqref{MSAssID} is only used to make sense of the equations in a stronger sense. 
 
 \item The assumption \eqref{eq:modscat_assump2} is used to guarantee the convergence \eqref{eq:nu-conv}. We note that this statement is slightly different from the one in Theorem \ref{thm:main_simple}, in that in \eqref{eq:modscat_assump2} we start with \emph{uniform} control of one additional moment in $p$ on the gradients and obtain \emph{uniform} (rather than local) convergence in \eqref{eq:nu-conv}. The proofs are easily adapted to establish the corresponding local statement under local assumptions as in Theorem \ref{thm:main_simple}.
 \item Our proof of Theorem \ref{ModScatThm} shows that control of higher moments (in both $p$ and $q$) as well as higher regularity can be propagated. For higher moments in $p$ this is explicitly done in Proposition \ref{prop:gl_bds}, and from this the propagation of moments in $q$ follows by the commutation relations \eqref{CommRel}. For higher regularity, by \eqref{CommRel} one needs control of derivatives of the electric field; these in turn can be directly bounded by derivatives of $\gamma$ via an adaptation of Lemmas \ref{LemControlEF} and \ref{lem:Ediff} (see e.g.\ Lemma \ref{ControlEPropReg} for one additional derivative). As a consequence, given more regularity and/or moments on a solution, the convergence \eqref{eq:nu-conv} can then be shown to hold in a correspondingly strengthened topology.
 \item The convergence \eqref{eq:nu-conv} implies the asymptotic dynamic \eqref{eq:modscat} of Theorem \ref{thm:main_simple}: Letting
 \begin{equation}
  \mathcal{A}:(s,q,p)\mapsto(s,q+ps+\lambda s\ln(s)E_0(q),p+\lambda\ln(s)E_0(q)),
 \end{equation}
 by $\mathcal{I}^2=Id$ there holds that
 \begin{equation}
  \gamma\circ\mathcal{A}(s,q,p)=\mu\circ (\mathcal{I}\circ\mathcal{A})(s,q,p)=\mu\left(\frac{1}{s},\frac{q}{s}+p+\lambda\ln(s)E_0(q),q\right),
 \end{equation}
 which gives \eqref{eq:modscat} with $\mu_\infty(x,v)=\gamma_0(v,x)$ by relabeling the arguments.
\end{enumerate}

\end{remark}

The proof of Theorem \ref{ModScatThm} makes frequent use of the fact that \eqref{NewVP} is a transport equation and we can propagate uniform bounds using the maximum principle along the characteristics. In particular, writing 
\begin{equation}
\begin{split}
\mathcal{L}:=\partial_s+p\cdot\nabla_q+\lambda s^{-1}E[\gamma]\cdot\nabla_p,\qquad\mathcal{L}f&=\partial_sf+\hbox{div}_{q,p}\left\{(p,\lambda s^{-1}E[\gamma](q))\cdot f\right\}
\end{split}
\end{equation}
we have that if $h$ is a strong solution in a neighborhood of $s=1$ to
\begin{equation}
 \mathcal{L}[h]=F(s,q,p)
\end{equation}
with $h(1)\in L^r_{q,p}$ for some $r\geq 1$, then since the transport field is divergence free there holds that
\begin{equation}\label{eq:transp_bd}
 \norm{h(s)}_{L^r_{q,p}}\leq \norm{h(1)}_{L^r_{q,p}}+\int_s^1 \norm{F(s^\prime)}_{L^r_{q,p}}ds^\prime
\end{equation}
for all $0\le s\le 1$ in the interval of existence.

\subsection{Commutation Relations}
Now consider a solution $\gamma$ to \eqref{NewVP}, i.e.\ $\L[\gamma]=0$.
In order to decide which equation we want to use, it will be convenient to compute some commutation relations: For any $m,n\in \{1,2,3\}$ we have
\begin{equation}\label{CommRel}
\begin{split}
\mathcal{L}[q^m\gamma]&=\L[q^m]\gamma=p^m\gamma,\qquad
\mathcal{L}[p^m \gamma]=\lambda s^{-1}E^m[\gamma]\gamma,\\
\mathcal{L}[\partial_{q^m}\gamma]&=\partial_{q^m}(\L[\gamma])-(\partial_{q^m}\L)[\gamma]=-\lambda s^{-1}\partial_{q^m}E^j[\gamma]\partial_{p^j}\gamma,\qquad
\mathcal{L}[\partial_{p^m}\gamma]=-\partial_{q^m}\gamma,
\end{split}
\end{equation}
and we also remark that
\begin{equation}\label{CommRel2}
\begin{split}
\mathcal{L}[p^m\partial_{q^n}\gamma]&=-\lambda s^{-1}p^m\partial_{q^n}E^j[\gamma]\partial_{p^j}\gamma+\lambda s^{-1}E^m[\gamma]\partial_{q^n}\gamma,\\
\mathcal{L}[p^m\partial_{p^n}\gamma]&=-p^m\partial_{q^n}\gamma+\lambda s^{-1}E^m[\gamma]\partial_{p^n}\gamma.
\end{split}
\end{equation}

\subsection{Bootstrap and Global Existence}
As a first step we see that so long as the electric field remains bounded, we can propagate all the moments we want.

\begin{lemma}\label{lem:E-mom}
Let $\gamma$ be a strong solution of \eqref{NewVP} on $T^\ast\le s\le 1$ with ``initial'' data $\gamma(s=1)=\gamma_1$. 
Assume that $\gamma_1$ satisfies for some $a\in\N$, $r\in [2,\infty]$ that
\begin{equation}\label{ControlID}
\begin{split}
\Vert \langle p
\rangle^a\gamma_1\Vert_{L^r_{q,p}}&\le \varepsilon_0,
\end{split}
\end{equation}
and that
\begin{equation}\label{UniformControlEF}
\begin{split}
\vert E(s,q)\vert\le D,\qquad T^\ast\le s\le 1.
\end{split}
\end{equation}
Then there holds that
\begin{equation}\label{GlobalBoundsMom1}
\begin{split}
\Vert \gamma(s)\Vert_{L^r_{q,p}}&\le\varepsilon_0,\\
\Vert \langle p\rangle^a\gamma(s)\Vert_{L^r_{q,p}}&\le\varepsilon_0+aD\varepsilon_0\langle\ln(s)\rangle^{a}.\\
\end{split}
\end{equation}
\end{lemma}

\begin{proof}
The proof follows by applying \eqref{CommRel} and \eqref{eq:transp_bd} inductively to $p^{\beta}\gamma$, $\beta\in \N_0^{3}$, with $\abs{\beta}\leq a$. 
\end{proof}

\begin{proposition}\label{prop:gl_bds}
 Let $0<\eps_0\leq\eps_1\ll 1$, and let $\gamma$ be a solution of \eqref{NewVP} on $T^\ast\le s\le 1$ with ``initial'' data $\gamma(s=1)=\gamma_1$ satisfying
 \begin{equation}\label{ControlID0}
  \norm{\gamma_1}_{L^2_{q,p}}+\norm{\gamma_1}_{L^\infty_{q,p}}\leq \eps_0.
 \end{equation}
\begin{enumerate}[leftmargin=*]
 \item\label{it:mom} (Moments and the electric field) If there holds that
\begin{equation}\label{ControlID1}
 \norm{\ip{p}\gamma_1}_{L^2_{q,p}}+\norm{\ip{p}^m\gamma_1}_{L^\infty_{q,p}}\leq \eps_1,\qquad m\ge 2
\end{equation}
 then the electric field $E(s)$ remains bounded and the solutions satisfies the bounds
 \begin{equation}\label{GlobalBounds_mom}
 \begin{aligned}
  \norm{\ip{p}\gamma(s)}_{L^2_{q,p}}&\lesssim \eps_1\ip{\ln s},\\
  \norm{\ip{p}^a\gamma(s)}_{L^\infty_{q,p}}&\lesssim \eps_1\ip{\ln s}^a, \qquad 0\leq a\leq m.
 \end{aligned} 
 \end{equation}
 Moreover, there exists $C>0$ (independent of $T^\ast$) such that for any $T^\ast\le s_1\le s_2\le 1$ there holds
 \begin{equation}\label{ETCont}
  \vert E(s_1,q)-E(s_2,q)\vert \le C \varepsilon_1^2\,\ip{\ln(s_1)}^4\ip{\ln(s_2-s_1)}^2(s_2-s_1).
 \end{equation}

 \item\label{it:deriv} (Derivatives) Assume additionally that for some $b\in \{0,1\}$ there holds that
 \begin{equation}\label{ControlID2}
   \norm{\ip{p}^b\nabla_{p,q}\gamma_1}_{L^\infty_{q,p}}\leq \eps_1.
 \end{equation}
 Then we have the bounds
 \begin{equation}\label{GlobalBounds_deriv}
 \begin{split}
  \norm{\ip{p}^a\nabla_p\gamma(s)}_{L^\infty_{q,p}}&\lesssim \eps_1\ip{\ln s}^a,\qquad 0\leq a\leq b,\\
  \norm{\ip{p}^a\nabla_q\gamma(s)}_{L^\infty_{q,p}}&\lesssim \eps_1\ip{\ln s}^{5+a},\qquad 0\leq a\leq b.
 \end{split}
 \end{equation}
\end{enumerate}
\end{proposition}

\begin{proof}
 We start by establishing claim \eqref{it:mom}. Let $C>0$ be a constant larger than all the implied constants appearing in Section \ref{SecE} and let $\varepsilon_1$ be small enough so that
\begin{equation}\label{AssEpsilon1}
 4C^2\varepsilon_1^2\le 1.
\end{equation}
We make the following bootstrap assumption: Let $I\subset [T^\ast,1]$ be such that for $s\in I$ there holds
 \begin{equation}\label{eq:btstrap}
 \begin{aligned}
 \norm{E(s)}_{L^\infty_q}&\leq 2C^2\eps_1^2.
 \end{aligned}
 \end{equation}
By the first line of \eqref{ControlEF} (with $A=1$) and the assumptions \eqref{ControlID0}, \eqref{ControlID1} we have that $1\in I\neq\emptyset$, and by continuity $I$ is closed in $[T^\ast,1]$. To establish the claim it then suffices to prove that \eqref{eq:btstrap} holds with strictly smaller constants, implying that $I$ is also open in $[T^\ast,1]$.
 
 To this end, note that by Lemma \ref{lem:E-mom} we have that for $0\leq a\leq m$,
 \begin{equation}\label{eq:p-mom}
  \norm{\ip{p}^a\gamma(s)}_{L^\infty_{q,p}}\leq \eps_1(1+aC^2\eps_1^2)\ip{\ln(s)}^a,\quad s\in I.
 \end{equation}
 By Lemma \ref{lem:Ediff} and \eqref{AssEpsilon1}, it then follows for $T^\ast\leq s_1\leq s_2\leq 1$ that
 \begin{equation}
  \norm{E(s_1)-E(s_2)}_{L^\infty_q}\leq 4C\eps_1^2\left[\ip{\ln(s_2-s_1)}^2+\varepsilon_1^2\langle\ln(s_2/s_1\rangle\right]\ip{\ln(s_1)}^4(s_2-s_1)
 \end{equation}
 and \eqref{ETCont} is proved. In particular, when $2^{-k}\le s_1\le s_2\le 2^{1-k}$, $k\ge 1$,
  \begin{equation}\label{ConvergenceEField}
  \norm{E(s_1)-E(s_2)}_{L^\infty_q}\leq 10C\eps_1^2 k^62^{-k}
 \end{equation} 
and since by \eqref{ControlEF} we have $\norm{E(1)}_{L^\infty_q}\leq 2C\eps_1^2$, we see that for $s\in I$
 \begin{equation}
  \norm{E(s)}_{L^\infty_q}\leq \norm{E(1)}_{L^\infty_q}+10C\varepsilon_1^2\sum_{k\ge 1} k^62^{-k} \ll C^2\eps_1^2,
 \end{equation}
 provided $C$ is large enough.
 
 To prove \eqref{it:deriv}, we use a similar bootstrap argument based on the assumptions
 \begin{equation}\label{DerBootstrap}
 \begin{split}
 \Vert \langle p\rangle^b\nabla_p\gamma(s)\Vert_{L^\infty_{q,p}}&\le  2C^4\varepsilon_1\langle \ln(s)\rangle^b,\\
 \Vert \langle p\rangle^b\nabla_q\gamma(s)\Vert_{L^\infty_{q,p}}&\le  2C^2 \varepsilon_1\langle \ln(s)\rangle^{5+b}.
 \end{split}
 \end{equation}
Using the commutation relations \eqref{CommRel} and \eqref{eq:transp_bd}, we deduce from \eqref{ControlID2} and \eqref{DerBootstrap} that
\begin{equation}\label{eq:dp_bd}
\begin{aligned}
 \norm{\nabla_p\gamma(s)}_{L^\infty_{q,p}}&\leq \norm{\nabla_p\gamma(1)}_{L^\infty_{q,p}}+\int_s^1 \norm{\nabla_q\gamma(s^\prime)}_{L^\infty_{q,p}} \,ds^\prime \leq C^3 2\eps_1,
\end{aligned} 
\end{equation}
provided $C>0$ is large enough. 

From the transport bounds  and the commutation relations \eqref{CommRel} we then deduce the estimate for $\nabla_q\gamma$: From \eqref{ControlEF'} we have under our assumptions \eqref{eq:btstrap} and
with \eqref{eq:p-mom} and \eqref{eq:dp_bd} that
\begin{equation}
\begin{aligned}
 \norm{\nabla_q\gamma(s)}_{L^\infty_{q,p}}&\leq \norm{\nabla_q\gamma(1)}_{L^\infty_{q,p}}+\int_s^1 \norm{\nabla_q E}_{L^\infty_{q}}\norm{\nabla_{p}\gamma}_{L^\infty_{q,p}} \,\frac{ds^\prime}{s^\prime}\\
 &\leq \eps_1+\int_s^1\left[\ip{\ln s^\prime}^{4}\norm{\gamma}_{L^2_{q,p}}^2+\norm{|p|^2\gamma}_{L^\infty_{q,p}}^2+\ip{\ln s^\prime}^{-\frac{6}{5}}\norm{\gamma}_{L^\infty_{q,p}}\norm{\nabla_q\gamma}_{L^\infty_{q,p}}\right]\norm{\nabla_p\gamma}_{L^\infty_{q,p}}\frac{ds^\prime}{s^\prime}\\
 &\leq \eps_1 + \ip{\ln s}^{5}\left(\eps_0^{2}+4\eps_1^2\right)\cdot 2C^4\eps_1+2C^4\eps_0\eps_1\int_s^1 \ip{\ln s^\prime}^{-\frac{6}{5}}\norm{\nabla_q\gamma(s^\prime)}_{L^\infty_{q,p}}\frac{ds^\prime}{s^\prime},
\end{aligned} 
\end{equation}
 so that by Gr\"onwall's lemma there holds that
 \begin{equation}\label{eq:dq_bd}
  \norm{\nabla_q\gamma(s)}_{L^\infty_{q,p}}\leq 10\eps_1\ip{\ln s}^5,
 \end{equation}
 provided $\eps_1$ is small enough. A similar argument using \eqref{CommRel2}, \eqref{eq:btstrap} and \eqref{DerBootstrap} shows that\begin{equation}\label{eq:pdp_bd}
\begin{aligned}
 \norm{p^m\partial_{p^n}\gamma(s)}_{L^\infty_{q,p}}&\leq \norm{\vert p\vert \nabla_{p}\gamma(1)}_{L^\infty_{q,p}}+\int_s^1 \norm{\vert p\vert \nabla_{q}\gamma(s^\prime)}_{L^\infty_{q,p}}ds^\prime +\int_s^1 \norm{E}_{L^\infty_{q}}\norm{\nabla_{p}\gamma(s^\prime)}_{L^\infty_{q,p}} \frac{ds^\prime}{s^\prime}\\
 &\leq C^3\eps_1\ip{\ln s}.
\end{aligned} 
\end{equation}
For the last bound, we see from \eqref{CommRel2} that we need a bound on the derivative of the electric field. Using \eqref{ControlEF'}, \eqref{eq:p-mom} and \eqref{eq:dq_bd}, we find that
\begin{equation}\label{ControlDerE}
\begin{aligned}
 \norm{\nabla_q E(s)}_{L^\infty_q}&\leq  C\left[\ip{\ln s}^{4}\norm{\gamma(s)}_{L^2_{q,p}}^2+\norm{p^2\gamma(s)}_{L^\infty_{q,p}}^2+\ip{\ln s}^{-\frac{6}{5}}\norm{\gamma(s)}_{L^\infty_{q,p}}\norm{\nabla_q\gamma(s)}_{L^\infty_{q,p}}\right]\\
 &\leq C\left(\eps_0^2+ 4\eps_1^2\right)\ip{\ln(s)}^4+10C\eps_0 \eps_1\ip{\ln s}^{5-\frac{6}{5}}\\
 &\leq 10 C\eps_1^2\ip{\ln s}^4,
\end{aligned} 
\end{equation}
so that \eqref{CommRel2} with \eqref{eq:transp_bd}, \eqref{eq:pdp_bd}, \eqref{eq:btstrap}, \eqref{eq:dq_bd} and \eqref{ControlDerE} gives
\begin{equation}
\begin{aligned}
 \norm{p^m \nabla_q\gamma(s)}_{L^\infty_{q,p}}&\leq \norm{\vert p\vert \nabla_q\gamma(1)}_{L^\infty_{q,p}}+\int_s^1 \left(\norm{\nabla_{q}E}_{L^\infty_q}\norm{\vert p\vert\nabla_{p}\gamma}_{L^\infty_{q,p}}+\norm{E}_{L^\infty_q}\norm{\nabla_{q}\gamma}_{L^\infty_{q,p}}\right)\,\frac{ds^\prime}{s^\prime}\\
 &\leq \eps_1+\eps_1\int_s^1 \left(10C^4\eps_1^3\ip{\ln s^\prime}^5+20C^2\eps_1^3\ip{\ln s^\prime}^5\right)\,\frac{ds^\prime}{s^\prime}\\
 &\leq \eps_1 \ip{\ln s}^6.
\end{aligned} 
\end{equation}
This closes the bootstrap \eqref{DerBootstrap}.
\end{proof}

\subsection{Asymptotic Behavior}
From \eqref{ETCont} we can deduce that the electric field has an asymptotic limit:
\begin{corollary}\label{cor:Econv}
 Let $\gamma$ be a solution of \eqref{NewVP} as in Proposition \ref{prop:gl_bds}, which is moreover defined for $s\in (0,1]$. Then the limit
\begin{equation*}
\begin{split}
 E_0(q):=\lim_{s\to0}E(s,q)
\end{split}
\end{equation*}
exists and is bounded
\begin{equation*}
\Vert E_0\Vert_{L^\infty_q}\lesssim \varepsilon_1^2.
\end{equation*}
In addition we have the following convergence rate: if $0\le s_1\le s_2\le 1,$ there holds
\begin{equation}\label{eq:E-rate}
\Vert E(s_1)-E(s_2)\Vert_{L^\infty_q}\lesssim\varepsilon^2\langle \ln(s_2)\rangle^6s_2.
\end{equation}
\end{corollary}

The rate of convergence \eqref{eq:E-rate} is linked to the topology we choose through the continuity equation \eqref{CE}. Our assumptions scale like ${\bf j}\in L^\infty_{q}$ and we obtain almost Lipschitz bounds in time.

\begin{proof}
It follows from \eqref{ConvergenceEField} in the proof above that $E(2^{-k})$ is Cauchy in $L^\infty_q$. Summing again \eqref{ConvergenceEField} gives \eqref{eq:E-rate}.
\end{proof}

Now we are in the position to give the proof of the modified scattering result:
\begin{proof}[Proof of Theorem \ref{ModScatThm}]
From Proposition \ref{prop:gl_bds} we obtain a global solution $\gamma$ on $(0,1]$, which satisfies \eqref{GlobalBounds_mom}, \eqref{GlobalBounds_deriv} and \eqref{eq:E-rate}. Next we define
\begin{equation}
 \nu(s,q,p):=\gamma(s,q+ps+\lambda s\ln(s)E_0(q),p+\lambda\ln(s)E_0(q)),
\end{equation}
which satisfies
\begin{equation}
\begin{aligned}
 \partial_s\nu&=\partial_s\gamma+p\cdot\nabla_q\gamma+\lambda(1+\ln(s))E_0(q)\cdot\nabla_q\gamma+\lambda s^{-1}E_0(q)\cdot\nabla_p\gamma\\
 &=\lambda E_0(q)\cdot\nabla_q\gamma+\lambda s^{-1} [E_0(q)-E(s,q+ps+\lambda s\ln(s)E_0(q)]\cdot\nabla_p\gamma,
\end{aligned} 
\end{equation}
where 
\begin{equation}
 s^{-1}\abs{E_0(q)-E(s,q+ps+\lambda s\ln(s)E_0(q)}\lesssim s^{-1}\abs{E_0(q)-E(s,q)}+\abs{p+\ln(s)E_0(q)}\norm{\nabla E(s)}_{L^\infty_q}.
\end{equation}
Hence by \eqref{ControlEF'}, Corollary \ref{cor:Econv} and \eqref{ControlDerE} we have that
\begin{equation}
 \norm{\partial_s\nu}_{L^\infty_{q,p}}\lesssim \eps_1^2\norm{\nabla_q\gamma}_{L^\infty_{q,p}}+\eps_1^2\ip{\ln(s)}^4\norm{p\nabla_p\gamma}_{L^\infty_{q,p}}+\eps_1^2\ip{\ln(s)}^6\norm{\nabla_p\gamma}_{L^\infty_{q,p}},
\end{equation}
which is integrable over $0\leq s\leq 1$.
\end{proof}

\bigskip
\section{Wave Operators and Cauchy Problem at Infinity}\label{sec:wave_ops}

Using the symplectic structure, the equations \eqref{NewVP} can be written as
\begin{equation}\label{PoissonBracket}
\gamma_s + \{\gamma,\mathcal H \}=0,\qquad \{f,g\}:=\nabla_qf\cdot\nabla_pg-\nabla_pf\cdot\nabla_qg
\end{equation}
with the Hamiltonian
\begin{equation}\label{OldH}
\mathcal H(s,q,p) := \frac{|p|^2}{2} + \lambda s^{-1}\phi(s,q),
\end{equation}
where $-\Delta \phi(s,q) = \int \gamma^2(s,q,p) dp$. We wish to find a new coordinate system $(w, z)$ for which the Cauchy problem at $s=0$ can be solved. For this, we
introduce the type-3 generating function\footnote{see e.g.\ \cite[Chapter 8]{MO2017}}
\begin{equation*}
S(s, w, p) := w\cdot p+ \frac{|p|^2}{2}s + \lambda \ln(s)\phi_0(w),
\end{equation*}
where $\phi_0 (q) = \phi(0,q)$. 
This gives rise to the canonical change of coordinates
\begin{align*}
 z&= \nabla_{w} S(w,p)=   p+  \lambda \ln(s)\nabla \phi_0(w), \\
q &=\nabla_p S(w,p) = w+ ps = w+ zs -  \lambda s \ln(s)\nabla \phi_0(w),
\end{align*}
or
\begin{equation}\label{ChangeOfCoord}
\begin{array}{rlcrl}
q=&w+sz-\lambda s\ln(s)\nabla\phi_0(w),&\qquad &w=&q-sp,\\
p=&z-\lambda\ln(s)\nabla\phi_0(w),&\qquad &z=&p+\lambda\ln(s)\nabla\phi_0(q-sp),
\end{array}
\end{equation}
with Jacobian matrix
\begin{equation}\label{JacobChangeCoord}
\frac{\partial (w,z)}{\partial (q,p)}=\begin{pmatrix} Id&-sId\\-\lambda\ln(s)\nabla E_0&Id+\lambda s\ln(s)\nabla E_0\end{pmatrix}.
\end{equation}
This corresponds to the new Hamiltonian
\begin{align*}
\mathcal K(s,w,z) &:= \mathcal H(s,q,p) - \partial_sS(s,w,p) = \lambda s^{-1} \big[ \phi(s,q) -  \phi_0(w)\big]
\end{align*}
and vector field
\begin{equation}\label{eq:dwK}
\begin{split}
\nabla_w\mathcal{K}&=-\lambda s^{-1}\left\{E(s,q)-E_0(w)\right\}-\lambda^2\ln(s)E(s,q)\cdot\nabla E_0(w),\qquad\nabla_z\mathcal{K}=-\lambda E(q),
\end{split}
\end{equation}
with the notation $E =-\nabla \phi$, $E_0 = -\nabla \phi_0$. It follows that
\begin{equation}\label{SigmaGamma}
\sigma(s,w,z) := \gamma(s,q,p)
\end{equation}
solves 
\begin{equation}\label{eq:sigma}
\begin{split}
0&=\partial_s\sigma + \{\sigma,\mathcal K\}=\partial_s\sigma + \nabla_w \sigma \cdot \nabla_z \mathcal K -  \nabla_z \sigma \cdot\nabla_w \mathcal K.\\
\end{split}
\end{equation}

\begin{remark}
We note that the new variables $(w,z)$ have a simple interpretation in terms of the original variables in \eqref{VP}:
$w=v$, $z=x-tv-\lambda\ln(t)E_0(v)$, which are the variables in which the modified scattering of Theorem \ref{thm:main_simple} and \cite{IPWW2020} is expressed.
\end{remark}

The main result of this section then is the following: 
\begin{theorem}\label{WaveOpsThm}
Assume that initial data $\sigma_0$ and $E_0$\footnote{these are linked through \eqref{DefESigma}} satisfy
\begin{equation}\label{eq:Einfty_bds}
\norm{E_0}_{W^{3,\infty}}\leq c_0^2,
\end{equation}
and
\begin{equation}\label{AssumptionsSigma0}
\begin{split}
\Vert \sigma_0\Vert_{L^2_{w,z}}+\Vert \langle z\rangle^5\sigma_0\Vert_{L^\infty_{w,z}}+\sum_{0\le m+n\le 2}\Vert \langle z\rangle^m\nabla_z^m\nabla_w^n\sigma_0\Vert_{L^\infty_{w,z}}&\le c_0.
\end{split}
\end{equation}
Then there exists $T^\ast=T^\ast(c_0)>0$ and a unique solution $\sigma\in C^0_{s}([0,T^\ast):L^2_{w,z})$ of \eqref{eq:sigma} with ``initial'' data $\sigma(s=0)=\sigma_0$, and such that $s\partial_s\sigma,\,\nabla_{w,z}\sigma\in C^0_{s,w,z}$. Moreover, for $0\leq s<T^\ast$ we have that
\begin{equation}\label{PropagationLawsSigma}
\begin{split}
\Vert \sigma(s)\Vert_{L^2_{w,z}}+ \norm{\ip{z}^5\sigma(s)}_{L^\infty_{w,z}}+\norm{\nabla_{w,z}\sigma(s)}_{L^\infty_{w,z}}\lesssim c_0,\qquad
\Vert \langle w,z\rangle^\ell\sigma(s)\Vert_{L^\infty_{w,z}}&\lesssim \Vert \langle w,z\rangle^\ell\sigma_0\Vert_{L^\infty_{w,z}},
\end{split}
\end{equation}
and if $c_0$ is sufficiently small we may take $T^\ast=1$.
\end{theorem}

The proof of Theorem \ref{WaveOpsThm} is given below in Section \ref{ssec:waveops_lwp}, after we have established some a priori estimates on the propagation of moments and derivatives for the system \eqref{eq:sigma} in Sections \ref{ssec:waveops_comm} and \ref{ssec:waveops_apriori}.

\subsection{Commutation Relations}\label{ssec:waveops_comm}
Writing $\mathfrak L = \partial_s + \{\cdot,\mathcal K\}$, for moments in $w,z$ we have the commutation relations
\begin{align}\label{CommutationRelationsAddedAtTheEnd}
\mathfrak L [w_j\sigma]  &=- \lambda E_j(s,q) \sigma,\\
\mathfrak L [z_j \sigma] &=    \left(\lambda s^{-1} [E_j(s,q) - E_{0,j}(w)] + \lambda^2 \ln(s) E(s,q) \cdot \nabla_{w} E_{0,j}(w)\right)\sigma.
\end{align}
For the derivatives, we have 
\begin{equation}\label{eq:GgradComm}
\begin{split}
\mathfrak{L}(\partial_1\sigma)&=\{\partial_1\mathcal{K},\sigma\},\qquad \mathfrak{L}(\partial_2\partial_1\sigma)=\{\partial_1\mathcal{K},\partial_2\sigma\}+\{\partial_2\mathcal{K},\partial_1\sigma\}+\{\partial_2\partial_1\mathcal{K},\sigma\}
\end{split}
\end{equation}
and this gives in block diagonal form
\begin{equation}\label{eq:gradComm}
\mathfrak L \begin{pmatrix} \nabla_{w} \sigma \\ \nabla_{z} \sigma  \end{pmatrix} = \begin{pmatrix} -\nabla_{w}\nabla_{z} \mathcal K & \nabla_{w}^2 \mathcal K \\
-\nabla_{z}^2 \mathcal K&  \nabla_{w}\nabla_{z} \mathcal K  \end{pmatrix}\begin{pmatrix} \nabla_{w} \sigma \\ \nabla_{z} \sigma   \end{pmatrix}, 
\end{equation}
with
\begin{equation}\label{FormulasHessianK}
\begin{split}
\nabla_{w^jw^k}^2 \mathcal K &= -\lambda s^{-1}\partial_j[ E_k(s,q) - E_{0,k}(w)]\\
&\quad -\lambda^2\ln(s)\left\{\partial_j E(s,q)\cdot\partial_k E_{0}(w) +\partial_kE(s,q)\cdot\partial_j E_{0}(w) +\partial_j\partial_k E_0(w)\cdot E(s,q)\right\}\\
&\quad -\lambda^3 s \ln(s)^2 \partial_k E_{0,a}(w)\partial_jE_{0,b}(w) \cdot \partial_a E_b(s,q),\\
\nabla_{w}\nabla_{z} \mathcal K &=-\lambda \nabla E(s,q)  - \lambda^2 s\ln(s) (\nabla E(s,q)\cdot \nabla) E_0(w), \\
\nabla_{z}^2 \mathcal K &= - \lambda s \nabla E(s,q).
\end{split}
\end{equation}
We note that the matrix $\nabla^2_{w,z}\mathcal{K}$ is {\it ill-conditioned}, and to mitigate this effect, we introduce a weight on the gradient:
\begin{equation}\label{DefCompression}
\begin{split}
\theta(s,z):=\frac{\langle z\rangle}{1+s\langle z\rangle},\qquad \frac{1}{2}\min\{\langle z\rangle,s^{-1}\}\le \theta(s,z)\le\min\{\langle z\rangle,s^{-1}\},
\end{split}
\end{equation}
which is linked to the vector field through \eqref{ControlHessianK} and satisfies nice differential equalities
\begin{equation*}
\begin{split}
\qquad\partial_s\theta=-\theta^2,\quad \nabla_z\theta=(z/\langle z\rangle^3)\cdot \theta^2.\\
\end{split}
\end{equation*}

\subsection{A Priori Estimates}\label{ssec:waveops_apriori}

The goal of this section is to bootstrap the following assumptions: Given $c_0$ as in \eqref{eq:Einfty_bds}, we assume that for $0\le s\le T(c_0)$ there holds that
\begin{equation}\label{MWOBootstarpSigma}
\begin{split}
\Vert \sigma(s)\Vert_{L^2_{w,z}}+\Vert \langle z\rangle^5\sigma(s)\Vert_{L^\infty_{w,z}}&\le A\le 4c_0,\\
\Vert \nabla_w\sigma(s)\Vert_{L^\infty_{w,z}}+\Vert\theta\nabla_z\sigma(s)\Vert_{L^\infty_{w,z}}&\le B\le 4c_0,\\
\Vert \nabla^2_{w,w}\sigma(s)\Vert_{L^\infty_{w,z}}+\Vert\theta\nabla^2_{w,z}\sigma(s)\Vert_{L^\infty_{w,z}}+\Vert\theta^2\nabla^2_{z,z}\sigma(s)\Vert_{L^\infty_{w,z}}&\le C\le 4c_0.
\end{split}
\end{equation}
As we will show below in Section \ref{ssec:apriori-e}, this implies in particular that
\begin{equation}\label{MWOBootstarpE}
\begin{split}
\vert \nabla_z\mathcal{K}(w,z)\vert&\le 2c_0^2,\\
 \vert \nabla_w\mathcal{K}(w,z)\vert&\le c_0^2\left\{\min\{s^{-1},\vert z\vert\}+\langle \ln(s)\rangle^{3}\right\},
\end{split}
\end{equation}
and that we have the derivative bounds
 \begin{equation}\label{ControlHessianK}
 \begin{split}
 \Vert \nabla_w\nabla_z\mathcal{K}\Vert_{L^\infty_{w,z}}+\Vert \theta\nabla_z\nabla_z\mathcal{K}\Vert_{L^\infty_{w,z}}+\Vert \theta^{-1}\nabla_w\nabla_w\mathcal{K}\Vert_{L^\infty_{w,z}}\le c_0^2\langle \ln(s)\rangle^{4}.
 \end{split}
 \end{equation}
These in turn can then be used to close the bootstrap for \eqref{MWOBootstarpSigma}, as in Section \ref{ssec:apriori-sigma}.

\subsubsection{A Priori Control on the Electric Field}\label{ssec:apriori-e}

Here we consider a particle density $\sigma\in C^0([0,s]:L^2_{w,z})$ such that $\sigma(0)=\sigma_0$ and which satisfies the bounds \eqref{MWOBootstarpSigma}. This creates an electric field $E(s)$ through the formula
\begin{equation}\label{DefESigma}
\begin{split}
E[\sigma](s,Q)=E(s,Q)=\iint \frac{Q-q(s,w,z)}{\vert Q-q(s,w,z)\vert^3}\sigma^2(s,w,z)dwdz=\iint \frac{Q-q}{\vert Q-q\vert^3}\gamma^2(s,q,p)dqdp
\end{split}
\end{equation}
where $\sigma$ and $\gamma$ are related through \eqref{SigmaGamma}. Simple bounds give uniformly in $R>0$:
\begin{equation*}
\begin{split}
\vert E(s_2,Q)-E(s_1,Q)\vert&\lesssim R\Vert \langle z\rangle^4\sigma\Vert_{L^\infty_{s,w,z}}^2+R^{-2}\Vert \sigma(s)\Vert_{L^\infty_sL^2_{w,z}}\Vert \sigma(s_2)-\sigma(s_1)\Vert_{L^2_{w,z}},
\end{split}
\end{equation*}
which ensures that $E$ is continuous in time. In the next lemma we adapt the bounds from Section \ref{SecE} to obtain stronger control as in \eqref{MWOBootstarpE}, \eqref{ControlHessianK}.

\begin{lemma}\label{lem:EDecay}

Let $\sigma\in C^0([0,T],L^2_{w,z})$ with $\sigma(0)=\sigma_0$ such that $\gamma$ satisfies the continuity equation \eqref{ContinuityEquation} and $E_0$ satisfies \eqref{eq:Einfty_bds}. Then there exists $T^*(c_0) \in (0,T]$ such that
\begin{enumerate}[label=(\roman*)]
 \item Assuming the first line of \eqref{MWOBootstarpSigma} holds, we obtain \eqref{MWOBootstarpE} for $0\le s\le T^\ast$.
 
 \item Assuming both lines of \eqref{MWOBootstarpSigma} hold, we obtain \eqref{ControlHessianK} for $0\le s\le T^\ast$.

\end{enumerate}
\end{lemma}

\begin{proof}

$(i)$ In order to use Lemma \ref{lem:Ediff}, we observe that
\begin{equation}\label{ComparisonNewOldCoord}
\vert p-z\vert\le c_0^2\langle \ln(s)\rangle,\qquad \vert q-w\vert\le s\vert z\vert+c_0^2s\langle\ln(s)\rangle
\end{equation}
and that the change of variable \eqref{ChangeOfCoord} preserves volume, so that
\begin{equation*}
\begin{split}
\Vert \gamma(s)\Vert_{L^r_{q,p}}&=\Vert \sigma(s)\Vert_{L^r_{w,z}},\\
\Vert \vert p\vert^\alpha \gamma(s)\Vert_{L^r_{q,p}}&\lesssim \Vert \vert z\vert^\alpha\gamma(s)\Vert_{L^r_{q,p}}+c_0^{2\alpha} \langle\ln(s)\rangle^\alpha\Vert\gamma(s)\Vert_{L^r_{q,p}}\lesssim c_0^{2\alpha} \langle\ln(s)\rangle^\alpha\Vert \sigma(s)\Vert_{L^r_{w,z}}+\Vert \vert z\vert^\alpha\sigma(s)\Vert_{L^r_{w,z}}.
\end{split}
\end{equation*}
In addition, since (see \eqref{JacobChangeCoord}) $\partial z/\partial p=Id+O(c_0s\langle\ln(s)\rangle)$ has bounded Jacobian, we see that
\begin{equation*}
\begin{split}
\Vert {\bf j}(s)\Vert_{L^\infty_q}&\le\Vert\int\left[\vert z\vert+c_0^2\langle\ln(s)\rangle\right]\sigma^2dz\Vert_{L^\infty_w}\le c_0^2\langle \ln(s)\rangle\Vert\langle z\rangle^3\sigma\Vert_{L^\infty_{w,z}}^2
\end{split}
\end{equation*}
and using \eqref{eq:EdiffVar}, we obtain that for $2^{-k-1}\le s_2\le s_1\le 2^{-k}$,
\begin{equation*}
\begin{split}
\Vert E(s_2,q)-E(s_1,q)\Vert_{L^\infty_q}&\lesssim \langle c_0\rangle^2 k^22^{-k}\Vert \langle z\rangle^{3}\sigma\Vert_{L^\infty_{s,w,z}}^2+2^{-2k}\Vert \sigma\Vert_{L^\infty_sL^2_{w,z}}^2\\
\end{split}
\end{equation*}
and summing we see that $E(2^{-k})$ is Cauchy in $L^\infty_{q}$ and that
 \begin{equation}\label{eq:EDecay}
\|E(s,q)-E_0(q)\|_{L^\infty_q} \lesssim \langle c_0\rangle^2  s\langle \ln(s)\rangle^2 \|\langle z\rangle^{3} \sigma\|_{L^\infty_{s,w,z}}^2+\langle c_0\rangle^2 s^2\| \sigma\|_{L^\infty_sL^2_{w,z}}^2.
\end{equation}
Using the formulas in  \eqref{eq:dwK}, the control on $\nabla_z\mathcal{K}$ follows directly, while we see that
\begin{equation*}
\begin{split}
\nabla_w\mathcal{K}(w,z)&=-\lambda s^{-1}\left\{E(s,q)-E_0(q)\right\}-\lambda s^{-1}\left\{E_0(q)-E_0(w)\right\}+O(c_0^2\langle\ln(s)\rangle)
\end{split}
\end{equation*}
and with \eqref{eq:EDecay}, \eqref{MWOBootstarpSigma}, \eqref{eq:Einfty_bds} and \eqref{ComparisonNewOldCoord}, we obtain \eqref{MWOBootstarpE}.

\medskip

$(ii)$ We want to use \eqref{eq:EdiffVarGrad}, which requires some additional control on the derivatives. From \eqref{JacobChangeCoord} we see that
\begin{equation*}
\nabla_q\gamma=\nabla_w\sigma-\lambda\ln(s)\nabla E_0\cdot\nabla_z\sigma
\end{equation*}
so that
\begin{equation*}
\begin{split}
\Vert \nabla_q\gamma\Vert_{L^r_{q,p}}&\lesssim c_0^2\langle\ln(s)\rangle\Vert \nabla_{w,z}\sigma\Vert_{L^r_{w,z}},
\end{split}
\end{equation*}
and
\begin{equation*}
\begin{split}
\Vert \nabla_q{\bf j}(s)\Vert_{L^\infty_q}&\lesssim c_0^2\langle \ln(s)\rangle\Vert\int\left[\vert z\vert+c_0^2\langle\ln(s)\rangle\right]\vert \sigma\vert \cdot\vert\nabla_{w,z}\sigma\vert dz\Vert_{L^\infty_w}\\
&\lesssim c_0^4\langle \ln(s)\rangle^2\left[\Vert\langle z\rangle^5\sigma\Vert_{L^\infty_{w,z}}^2+\Vert \nabla_{w,z}\sigma\Vert_{L^\infty_{w,z}}^2\right].
\end{split}
\end{equation*}

For $2^{-k-1}\le s_2\le s_1\le 2^{-k}$ this gives by \eqref{eq:EdiffVarGrad} that
 \begin{equation*}
 \begin{aligned}
 \|\nabla_q E(s_2,q)-\nabla_q E(s_1,q)\|_{L^\infty_q} &\lesssim \langle c_0\rangle^4 k^32^{-k}\cdot \Big(\|\langle z\rangle^{5} \sigma\|_{L^\infty_{s,w,z})} ^2+\| \nabla_{w,z} \sigma\|_{L^\infty_{s,w,z}}^2\Big) \\
 &\quad+ c_0^{10}2^{-3k/2}\left[\Vert\langle z\rangle^5\sigma\Vert_{L^\infty_{s,w,z}}^2+\Vert \sigma\Vert_{L^\infty_sL^2_{w,z}}^2\right],
\end{aligned}
\end{equation*}
and applying similar arguments as before we obtain
 \begin{equation}\label{eq:GradEDecay}
 \begin{aligned}
 \|\nabla_q E(s,q)-\nabla_q E_0(q)\|_{L^\infty_q} &\lesssim \langle c_0\rangle^2 s \langle \ln(s)\rangle^3\cdot \Big(\|\langle z\rangle^{5} \sigma\|_{L^\infty_{s,w,z}} ^2+\| \nabla_{w,z} \sigma(s,w,z)\|_{L^\infty_{s,w,z}}^2\Big) \\
 &\quad+ c_0^{10}s^\frac{3}{2}\left[\Vert\langle z\rangle^5\sigma\Vert_{L^\infty_{s,w,z}}^2+\Vert\sigma\Vert_{L^\infty_sL^2_{w,z}}^2\right]\\
 &\le c_0^2s\langle \ln(s)\rangle^4
\end{aligned}
\end{equation}
up to choosing $T(c_0)>0$ small enough. Using the formulas in \eqref{FormulasHessianK} we directly see that
\begin{equation*}
\begin{split}
\theta \vert \nabla^2_{z,z}\mathcal{K}\vert&\le \vert \nabla E\vert\cdot s\min\{s^{-1},\vert z\vert\}\le c_0^2,\\
\vert \nabla^2_{w,z}\mathcal{K}\vert&\le \vert\nabla E\vert(1+s\langle \ln(s)\rangle\vert\nabla E_0\vert)\le 2c_0^2.
\end{split}
\end{equation*}
Moreover, using \eqref{eq:Einfty_bds} and \eqref{eq:GradEDecay} we find that, up to choosing $T(c_0)>0$ smaller,
\begin{equation*}
\begin{split}
\vert\nabla^2_{w,w}\mathcal{K}\vert&\le s^{-1}\vert \nabla E_0(q)-\nabla E_0(w)\vert+s^{-1}\vert\nabla E(s,q)-\nabla E_0(w)\vert\\
&\quad+\langle\ln(s)\rangle\left[2\vert\nabla E_0\vert^2\vert\nabla E\vert+\vert\nabla^2E_0\vert\vert E\vert\right]+s\langle\ln(s)\rangle^2\vert\nabla E_0\vert^2\vert\nabla E\vert\\
&\le c_0^2\min\{s^{-1},\vert z\vert\}+c_0^2\langle\ln(s)\rangle^4,
\end{split}
\end{equation*}
from which we deduce \eqref{ControlHessianK}.
\end{proof}

\subsubsection{A Priori Estimates on the Particle Density}\label{ssec:apriori-sigma}
Here we close the bootstrap of \eqref{MWOBootstarpSigma}:

\begin{lemma}\label{BootstrapSigmaLem}
Assume that $\sigma\in C^0([0,T],L^2_{w,z})$ satisfies \eqref{eq:sigma}, for some Hamiltonian $\mathcal{K}$ (not necessarily related to $\sigma$) satisfying \eqref{MWOBootstarpE} and \eqref{ControlHessianK}. If $\sigma_0=\sigma(0)$ satisfies \eqref{AssumptionsSigma0}, there exists $T(c_0)>0$ such that \eqref{MWOBootstarpSigma} holds for $A=B=C=2c_0$.
\end{lemma}

\begin{proof}
We first close the bootstrap for $A$, then for $A,B$. Finally we adapt the argument for $A,B,C$. The control follows from the commutation relations (compare with \eqref{eq:gradComm}):
\begin{equation}\label{CommutationRelationsSigma}
\begin{split}
\mathfrak{L}(\langle z\rangle^m\sigma)&=\sigma\{\langle z\rangle^m,\mathcal{K}\}=-m \sigma\langle z\rangle^{m-2} z\cdot \nabla_w\mathcal{K} ,\\
\mathfrak{L}(\theta\nabla_z\sigma)&=(\mathfrak{L}\ln\theta)\cdot\theta \nabla_z\sigma+\theta\{\nabla_z\mathcal{K},\sigma\}=(\mathfrak{L}\ln\theta)\cdot\theta \nabla_z\sigma+\theta\nabla_z\sigma\cdot \nabla_w\nabla_z\mathcal{K}-\nabla_w\sigma\cdot \theta\nabla_z\nabla_z\mathcal{K},\\
\mathfrak{L}(\nabla_w\sigma)&=\{\nabla_w\mathcal{K},\sigma\}=\theta\nabla_z\sigma\cdot \theta^{-1}\nabla_w\nabla_w\mathcal{K}-\nabla_w\sigma\cdot\nabla_z\nabla_w\mathcal{K}.
\end{split}
\end{equation}
As in \eqref{eq:transp_bd}, we find that
\begin{equation*}
\begin{split}
\Vert \langle z\rangle^m\sigma(s)\Vert_{L^r_{w,z}}&\le \Vert \langle z\rangle^m\sigma_0\Vert_{L^r_{w,z}}+m\int_0^s\Vert \langle z\rangle^{-1}\nabla_{w}\mathcal{K}(s^\prime)\Vert_{L^\infty_{w,z}}\Vert \langle z\rangle^m\sigma(s^\prime)\Vert_{L^r_{w,z}}ds^\prime,
\end{split}
\end{equation*}
and we can easily propagate the first line of \eqref{MWOBootstarpSigma}.

\medskip

For the derivatives, we also need to control $\theta$.
On the one hand, we can bound from above (note that $\mathfrak{L}(\ln\theta)$ can be very negative)
\begin{equation*}
\begin{split}
&\mathfrak{L}(\ln\theta)=-(1+\frac{z}{\langle z\rangle^3}\nabla_w\mathcal{K})\theta\lesssim c_0^2+c_0^2\langle \ln(s)\rangle^3\\
\end{split}
\end{equation*}
and we deduce from \eqref{CommutationRelationsSigma}, \eqref{eq:transp_bd} and \eqref{ControlHessianK} that
\begin{equation*}
\begin{split}
\Vert \theta\nabla_z\sigma(s)\Vert_{L^r_{w,z}}&\le\Vert \theta\nabla_z\sigma_0\Vert_{L^r_{w,z}}+ c_0^2\int_0^s\langle\ln(s^\prime)\rangle^4\left\{\Vert \theta\nabla_z\sigma(s^\prime)\Vert_{L^r_{w,z}}+\Vert\nabla_w\sigma(s^\prime)\Vert_{L^r_{w,z}}\right\}ds^\prime,\\
\Vert \nabla_w\sigma(s)\Vert_{L^r_{w,z}}&\le\Vert \nabla_w\sigma_0\Vert_{L^r_{w,z}}+ c_0^2\int_0^s\langle\ln(s^\prime)\rangle^4\left\{\Vert \theta\nabla_z\sigma(s^\prime)\Vert_{L^r_{w,z}}ds+\Vert\nabla_w\sigma(s^\prime)\Vert_{L^r_{w,z}}\right\}ds^\prime,\\
\end{split}
\end{equation*}
and this allows us to propagate the second line of \eqref{MWOBootstarpSigma} for short time.

\medskip

We now propagate higher order derivatives to bound the bootstrap for $C$. First by interpolation in \eqref{MWOBootstarpSigma}, we observe that
\begin{equation*}
\begin{split}
\Vert \langle z\rangle^{2.1}\nabla_{w,z}\sigma\Vert_{L^\infty_{w,z}}&\le A+C.
\end{split}
\end{equation*}
We will use the weight $\theta$ to control the $\partial_z$ derivatives. Using \eqref{eq:GgradComm}, we find that
\begin{equation*}
\begin{split}
\mathfrak{L}(\partial_{w^j}\partial_{w^k}\sigma)&=\theta^{-1}\nabla_w\partial_{w^j}\mathcal{K}\cdot (\theta\nabla_z\partial_{w^k}\sigma)+\theta^{-1}\nabla_w\partial_{w^k}\mathcal{K}\cdot(\theta\nabla_z\partial_{w^j}\sigma)-\nabla_z\partial_{w^j}\mathcal{K}\cdot\nabla_w\partial_{w^k}\sigma\\
&\quad-\nabla_z\partial_{w^k}\mathcal{K}\cdot\nabla_w\partial_{w^j}\sigma+\theta^{-1}\nabla_w\partial_{w^j}\partial_{w^k}\mathcal{K}\cdot(\theta\nabla_z\sigma)-\nabla_z\partial_{w^j}\partial_{w^k}\mathcal{K}\cdot\nabla_w\sigma,\\
\mathfrak{L}(\theta\partial_{z^j}\partial_{w^k}\sigma)&=\mathfrak{L}(\ln\theta)\cdot\theta\partial_{z^j}\partial_{w^k}\sigma+\theta^{-1}\nabla_w\partial_{w^k}\mathcal{K}\cdot(\theta^2\nabla_z\partial_{z^j}\sigma)-\nabla_z\partial_{w^k}\mathcal{K}\cdot(\theta\nabla_w\partial_{z^j}\sigma)\\
&\quad+\nabla_w\partial_{z^j}\mathcal{K}\cdot (\theta\nabla_z\partial_{w^k}\sigma)-(\theta\nabla_z\partial_{z^j}\mathcal{K})\cdot\nabla_w\partial_{w^k}\sigma\\
&\quad+\nabla_w\partial_{z^j}\partial_{w^k}\mathcal{K}\cdot(\theta\nabla_z\sigma)-\theta\nabla_z\partial_{z^j}\partial_{w^k}\mathcal{K}\cdot\nabla_w\sigma,\\
\mathfrak{L}(\theta^2\partial_{z^j}\partial_{z^k}\sigma)&=2\mathfrak{L}(\ln\theta)\cdot\theta^2\partial_{z^j}\partial_{z^k}\sigma+\nabla_w\partial_{z^k}\mathcal{K}\cdot(\theta^2\nabla_z\partial_{z^j}\sigma)-\theta\nabla_z\partial_{z^k}\mathcal{K}\cdot(\theta\nabla_w\partial_{z^j}\sigma)\\
&\quad+\nabla_w\partial_{z^j}\mathcal{K}\cdot (\theta^2\nabla_z\partial_{z^k}\sigma)-(\theta\nabla_z\partial_{z^j}\mathcal{K})\cdot (\theta\nabla_w\partial_{z^k}\sigma)\\
&\quad+\theta\nabla_w\partial_{z^j}\partial_{z^k}\mathcal{K}\cdot(\theta\nabla_z\sigma)-\theta^2\nabla_z\partial_{z^j}\partial_{z^k}\mathcal{K}\cdot\nabla_w\sigma,
\end{split}
\end{equation*}
and we can proceed as for the case of one derivative once we control the new terms
\begin{equation}\label{3DerK}
\begin{split}
\Vert \theta^{-1}\nabla^3_{w,w,w}\mathcal{K}\Vert_{L^\infty_{w,z}}+\Vert \nabla^3_{w,w,z}\mathcal{K}\Vert_{L^\infty_{w,z}}+\Vert \theta\nabla^3_{w,z,z}\mathcal{K}\Vert_{L^\infty_{w,z}}+\Vert \theta^2\nabla^3_{z,z,z}\mathcal{K}\Vert_{L^\infty_{w,z}}&\le c_0^2\langle\ln(s)\rangle^5.
\end{split}
\end{equation}
It remains to prove \eqref{3DerK}. Starting from
\begin{equation*}
\begin{split}
\nabla_z\mathcal{K}=-\lambda E(q),\qquad \frac{\partial q^k}{\partial z^j}=s\delta_j^k,\qquad\frac{\partial q^k}{\partial w^j}=\delta_j^k-\lambda s\ln(s)\partial_j\partial_k\phi_0(w)
\end{split}
\end{equation*}
we deduce
\begin{equation*}
\begin{split}
\theta^2\vert \nabla^3_{z,z,z}\mathcal{K}\vert&=(s\theta)^2\vert\nabla^2E(q)\vert,\\
\theta\vert\nabla^3_{w,z,z}\mathcal{K}\vert&\le (s\theta)\cdot\left[1+s\langle \ln(s)\rangle\vert\nabla E_0\vert\right]\cdot \vert\nabla^2E(q)\vert,\\
\vert\nabla^3_{w,w,z}\mathcal{K}\vert&\le\left[1+s\langle \ln(s)\rangle\vert\nabla E_0\vert\right]^2\cdot \vert\nabla^2E(q)\vert+\left[1+s\langle \ln(s)\rangle\vert\nabla E_0\vert\right]\cdot\left[1+s\langle \ln(s)\rangle\vert\nabla^2 E_0\vert\right]\cdot \vert\nabla E(q)\vert,\\
\end{split}
\end{equation*}
and finally, from \eqref{FormulasHessianK}, we obtain that
\begin{equation*}
\begin{split}
\theta^{-1}\vert\nabla^3_{w,w,w}\mathcal{K}\vert&\le \left[s^{-1}+\langle \ln(s)\rangle\cdot \vert\nabla E_0\vert\right]\cdot \vert\nabla^2E(s,q)-\nabla^2E_0(w)\vert\\
&\quad+\langle\ln(s)\rangle\cdot \left[\vert\nabla^2E\vert\cdot\vert\nabla E_0\vert+\vert\nabla E\vert\cdot \vert\nabla^2E_0\vert+\vert\nabla^3E_0\vert\cdot\vert E\vert\right]\\
&\quad+s\langle\ln(s)\rangle^2\cdot\left[\vert\nabla^2 E\vert\cdot\vert\nabla E_0\vert^2+\vert\nabla E\vert \cdot\vert\nabla E_0\vert\cdot\vert\nabla^2E_0\vert\right]\\
&\quad+s^2\langle\ln(s)\rangle^3\cdot\left[\vert \nabla^2 E\vert\cdot\vert\nabla E_0\vert^3\right].
\end{split}
\end{equation*}

Independently, we find that
\begin{equation*}
\begin{split}
\Vert\nabla^2{\bf j}(s)\Vert_{L^\infty_q} &\lesssim c_0^2\langle \ln(s)\rangle^2\Vert\int\left[\vert z\vert+c_0^2\langle\ln(s)\rangle\right]\cdot\left[\vert \sigma\vert \cdot\vert\nabla^2_{w,z}\sigma\vert+\vert\nabla_{w,z}\sigma\vert^2\right] dz\Vert_{L^\infty_w}\\
&\lesssim c_0^4\langle\ln(s)\rangle^3\left[\Vert \langle z\rangle^5\sigma\Vert_{L^\infty_{w,z}}\Vert\nabla^2_{w,z}\sigma\Vert_{L^\infty_{w,z}}+\Vert\langle z\rangle^{2.1}\nabla_{w,z}\sigma\Vert_{L^\infty_{w,z}}^2\right].
\end{split}
\end{equation*}

Now using Lemma \ref{ControlEPropReg} and the bootstrap assumptions, we obtain that
\begin{equation*}
\begin{split}
\Vert\nabla^2E(s,q)-\nabla^2E_0(w)\Vert_{L^\infty_{w,z}}&\le c_0^2\langle \ln(s)\rangle^5+c_0^2\min\{s^{-1},\vert z\vert\}
\end{split}
\end{equation*}
which easily leads to \eqref{3DerK}.
\end{proof}

\subsection{Local Solutions}\label{ssec:waveops_lwp}

We construct local solutions for the singular equation \eqref{eq:sigma} via Picard iteration.

\begin{proof}[Proof of Theorem \ref{WaveOpsThm}]
We proceed in two steps.
\paragraph{{\bf Step 1: A priori estimates}.} We construct a sequence of approximate solutions on a time interval $[0,T]$ (with $T>0$ to be chosen later) via Picard iteration: We define $\sigma_{(0)}(s,w,z):=\sigma_0(w,z)$, and given $\sigma_{(n)}\in C^0_s([0,T],C^1_{w,z})$ satisfying \eqref{MWOBootstarpSigma} with $A=B=C=4c_0$, we let $\sigma_{(n+1)}\in C^0_s([0,T], C^1_{w,z})$  be the  solution of
 \begin{equation}
 \begin{aligned}
  &\partial_s\sigma_{(n+1)}+\{\sigma_{(n+1)},\K_n\}=0,\qquad \sigma_{(n+1)}(0)=\sigma_0,\\
  &\K_n:=\lambda s^{-1}(\phi_n(s,q)-\phi_0(w)),\qquad -\Delta\phi_n=\int\gamma_{(n)}^2(s,q,p)dp,
 \end{aligned} 
 \end{equation}
 where $\gamma_{(n)}$ and $\sigma_{(n)}$ are related through \eqref{SigmaGamma}. Using Lemma \ref{lem:EDecay}, we see that $\mathcal{K}_n$ satisfies \eqref{MWOBootstarpE} and \eqref{ControlHessianK}. Using Lemma \ref{BootstrapSigmaLem}, we see that $\sigma_{(n+1)}$ satisfies \eqref{MWOBootstarpSigma} with $A=B=C=2c_0$. We deduce that \eqref{MWOBootstarpSigma} holds uniformly in $n$ with $A=B=C=2c_0$ on a fixed time interval $0\le s\le T(c_0)$.
 
 In addition using the commutation relations \eqref{CommutationRelationsAddedAtTheEnd}, we easily propagate \eqref{PropagationLawsSigma} uniformly in $n$.
 
 \paragraph{{\bf Step 2: Contraction in $L^\infty_{s,w,z}$}} 
 Let
\begin{equation*}
\begin{split}
\delta_{(n)}&:=\sigma_{(n+1)}-\sigma_{(n)},\qquad\delta\mathcal{K}_{(n)}:=\K_n-\K_{n-1},\qquad
\mathfrak{L}_n:=\partial_s+\left\{\cdot,\mathcal{K}_n\right\},\qquad\delta\mathfrak{L}_n=\{\cdot,\delta\mathcal{K}_{(n)}\},
\end{split}
\end{equation*}
so that
\begin{equation}\label{EvolEqDelta}
\begin{split}
\mathfrak{L}_n\delta_{(n)}&=\delta\mathfrak{L}_n\sigma_{(n)},
\end{split}
\end{equation}
and we can express
\begin{equation*}
\begin{split}
\nabla_z\delta\mathcal{K}_{(n)}&=-\lambda(E_n(s,q)-E_{n-1}(s,q)),\\
\nabla_w\delta\mathcal{K}_{(n)}&=-\lambda s^{-1}(E_n(s,q)-E_{n-1}(s,q))-\lambda^2\ln(s)(E_n(s,q)-E_{n-1}(s,q))\cdot\nabla E_0(q).
\end{split}
\end{equation*} 
Invoking the uniform bounds for $\sigma_{(n)}$ we will prove below that
\begin{equation}\label{ControlELWP}
\begin{split}
\Vert \nabla_{w,z}\delta\mathcal{K}_{(n)}(s)\Vert_{L^\infty_{w,z}}&\le  c_0\langle \ln(s)\rangle^6\norm{\delta_{(n-1)}(s)}_{L^\infty_{w,z}}.
\end{split}
\end{equation}
In combination with \eqref{EvolEqDelta}, we find that
\begin{equation*}
\begin{split}
\Vert \delta_{(n)}(s)\Vert_{L^\infty_{w,z}}&\lesssim \int_0^s\Vert\nabla_{w,z}\delta\mathcal{K}_{(n)}(s^\prime)\Vert_{L^\infty_{w,z}}\Vert \nabla_{w,z}\sigma_{(n)}(s^\prime)\Vert_{L^\infty_{w,z}}ds^\prime\lesssim c_0^2\int_0^s\langle\ln(s^\prime)\rangle^6\norm{\delta_{(n-1)}(s^\prime)}_{L^\infty_{w,z}}ds^\prime,
\end{split}
\end{equation*}
from which we deduce that, possibly taking $T(c_0)>0$ smaller,  $\sigma_{(n)}$ form a Cauchy sequence in $L^\infty_{s,w,z}$, and thus $\sigma_{(n)}\to \sigma\in L^\infty_{s,w,z}$ as $n\to \infty$. Interpolation gives convergence in the other topologies. In particular
\begin{equation}
 \norm{\nabla_{w,z}\delta_{(n)}}_{L^\infty_{s,w,z}}\lesssim\norm{\delta_{(n)}}_{L^\infty_{s,w,z}}^{1/2}\left[\norm{\nabla_{w,z}^2\sigma_{(n+1)}}_{L^\infty_{s,w,z}}+\Vert\nabla^2_{w,z}\sigma_{(n)}\Vert_{L^\infty_{s,w,z}}\right]^{1/2}
\end{equation}
so that $\sigma_{(n)}$ is Cauchy in $C^0_sC^1_{w,z}$ and the other bounds follow by Fatou's Lemma or by conservation. In particular, \eqref{PropagationLawsSigma} follows by pointwise convergence. Finally we note that if $c_0$ is sufficiently small, the arguments give a contraction for any $T\leq 2$.

\medskip

It remains to show \eqref{ControlELWP}. The main point is that $E$ is quadratic in $\gamma$, so that in the estimates for $\delta\mathcal{K}_{(n)}$ we can always factor out the difference $\delta_{(n)}$ in $L^\infty_{w,z}$. The bound on $\nabla_z\delta \mathcal{K}_{(n)}$ follows from adaptation to Lemma \ref{LemControlEF} and this also allows to control all but the first terms in $\nabla_w\delta \mathcal{K}_{(n)}$. These then follow from \eqref{eq:EdiffVar} using the difference continuity equation
\begin{equation*}
\begin{split}
\partial_s(E_n-E_{n-1})+\nabla\Delta^{-1}\hbox{div}_q\left\{\delta{\bf j}_n\right\},\qquad\delta{\bf j}_n=\int_{\mathbb{R}^3_p}(\gamma_{n}+\gamma_{n-1})(\gamma_n-\gamma_{n-1})\cdot pdp
\end{split}
\end{equation*}
with
\begin{equation*}
\begin{split}
\Vert \delta{\bf j}_n\Vert_{L^\infty_{q}}&\lesssim \Vert \delta_{(n-1)}\Vert_{L^\infty_{w,z}}\cdot\left[\Vert\langle p\rangle^5\gamma_n\Vert_{L^\infty_{q,p}}+\Vert\langle p\rangle^5\gamma_{n-1}\Vert_{L^\infty_{q,p}}\right]
\end{split}
\end{equation*}
and simple adaptations of Lemma \ref{lem:Ediff}.
\end{proof}

Finally, we prove the main theorem.
\begin{proof}[Proof of Theorem \ref{thm:main_simple}]

For \ref{it:modscat}, using \eqref{TransformationMuGamma}, the assumptions \eqref{eq:modscat_assump} lead to \eqref{MSAssID} in Theorem \ref{ModScatThm} and the local convergence is easily adapted (see Remark \ref{RemarkModScatThm}). For \ref{it:waveops}, assumptions \eqref{eq:waveop_assump} lead to \eqref{eq:Einfty_bds}-\eqref{AssumptionsSigma0} and the conclusion follows from that of Theorem \ref{WaveOpsThm}. Finally, for \ref{it:ScatOp}, we can apply Theorem \ref{WaveOpsThm} to $\mu_{-\infty}(a,-b)$ to get, using \eqref{SigmaGamma}, \eqref{TransformationMuGamma} and \eqref{PropagationLawsSigma} a solution for $-\infty< t\le-1$ such that
\begin{equation*}
\begin{split}
\Vert \mu(-1)\Vert_{L^2_{x,v}}+\Vert \langle x,v\rangle^5\mu(-1)\Vert_{L^\infty_{x,v}}+\Vert\nabla_{x,v}\mu(-1)\Vert_{L^\infty_{x,v}}\lesssim \varepsilon.
\end{split}
\end{equation*}
By local existence, we can extend these bounds for $-1\le t\le 1$, at which point we can simply apply \ref{it:modscat}.
\end{proof}

\bigskip

\bibliographystyle{abbrv}
\bibliography{VP_PCT.bib}

\end{document}